\documentclass[11pt,a4paper,leqno]{article}
\usepackage[latin1]{inputenc}
\usepackage[T1]{fontenc}
\usepackage[english]{babel}
\usepackage{graphicx}
\usepackage{amsfonts,amssymb,amsbsy,latexsym,a4wide,xspace}
\usepackage{amsmath,delarray,enumerate,calc,amsthm}
\usepackage{bbm}
\usepackage{mathrsfs}
\usepackage{txfonts}
\usepackage{paralist}

\usepackage{authblk}

\usepackage{tikz}
\usetikzlibrary{matrix}

\usepackage{url}

\newtheorem{Th}{Theorem}[section]
\newtheorem{Prop}[Th]{Proposition}
\newtheorem{Lemma}[Th]{Lemma}
\newtheorem{Conj}{Conjecture}
\theoremstyle{definition}
\newtheorem{Remark}[Th]{Remark}

\newtheorem{Cor}[Th]{Corollary}

\newcommand{\cF}{{\mathcal F}}
\newcommand{\cf}{{\mathcal F}}

\newcommand{\cm}{{\mathcal M}}

\newcommand{\cM}{{\mathcal M}}

\newcommand{\N}{{\mathbbm N}}

\newcommand{\Z}{{\mathbbm Z}}
\newcommand{\C}{{\mathbbm C}}

\newcommand{\1}{{\mathbbm 1}}
\newcommand{\raz}{{\mathbbm 1}}
\renewcommand{\P}{{\mathbbm P}}
\newcommand{\PP}{{\mathbbm P}}
\newcommand{\va}{\varphi}

\newcommand{\lcm}{\operatorname{lcm}}

\newcommand{\mob}{\boldsymbol{\mu}}
\newcommand{\vep}{\varepsilon}
\newcommand{\beq}{\begin{equation}}

\newcommand{\eeq}{\end{equation}}
\newcommand{\ov}{\overline}
\renewcommand{\mod}{\text{ mod }}

\newtheorem {lemma}{Lemma}[section]

\newtheorem {bemerkung}{Remark}[section]
\newtheorem{proposition}{Proposition}[section]
\newtheorem {corollary}{Corollary}[section]
\newtheorem{beispiel}{Example}[section]
\newtheorem{frage}{Question}[section]

\newenvironment{remark} {\begin{bemerkung} \normalfont }{\end{bemerkung}}

\title{Dynamics of $\mathscr{B}$-free systems generated by Behrend sets. I}

\author{S. Kasjan, M. Lema\'nczyk\thanks{Research supported by  Narodowe Centrum Nauki grant UMO-2019/33/B/ST1/00364}, \ S. Zuniga Alterman$^*$}

\begin{document}
\maketitle
\begin{center}{\em Dedicated to the memory of Professor Andrzej Schinzel}\end{center}

\begin{abstract} We study the complexity of $\mathscr{B}$-free subshifts which are proximal and of zero entropy. Such subshifts are generated by Behrend sets. The complexity is shown to achieve any subexponential growth and is estimated for some classical subshifts (prime and semiprime subshifts). We also show that $\mathscr{B}$-admissible subshifts are transitive only for coprime sets $\mathscr{B}$ which allows one to characterize dynamically the subshifts generated by the Erd\"os sets.
\end{abstract}
\section{Introduction}
\subsection{General overview and motivations}
In this paper we mainly study subshifts\footnote{\label{f:subshift} Given $y\in\{0,1\}^{\Z}$, by $(X_y,S)$ we denote the {\em subshift generated by} $y$, i.e. $X_y:=\ov{\{S^jy:\: j\in\Z\}}$, where $S$ stands for the left shift, $S((z_n)_{n\in\Z})=(z_{n+1})_{n\in\Z}$. We also consider $y$ which are only one-sided, i.e.\ \textcolor{black}{$y\in\{0,1\}^{\N\cup\{0\}}$}. Then $X_y$ means the subshift obtained from the symmetrized $y$: $y(0)=0$ and $y(-n)=y(n)$. For example, if $A\subset\N\textcolor{black}{=\{1,2,\ldots\}}$ then, $X_{\raz_{A}}$ in fact means $X_{\raz_{A\cup(-A)}}$.} $(X_{\raz_{\cf}},S)$ of the full shift $(\{0,1\}^{\Z},S)$, where $1\in\cf\subset \Z\setminus\{0\}$ is symmetric ($x\in\cf$ implies $-x\in\cf$)
 and satisfies:
\beq\label{behr1}
\mbox{$\cf$ is closed under  taking divisors}\eeq
and
\beq\label{behr2}
\mbox{the natural density $d(\cf):=\lim_{N\to\infty}\frac1N\sum_{n\leq N}\raz_{\cf}(n)$ of $\cf$ exists and equals zero.}\eeq
Let us explain the first condition. It is not hard to see that for each $\cf$ satisfying~\eqref{behr1} there is a subset $\mathscr{B}\subset \N\setminus\{1\}$ such that $\cf=\cf_{\mathscr{B}}$, where $\cf_{\mathscr{B}}$ denotes the set of $\mathscr{B}$-{\em free numbers}, i.e., of numbers with no divisor in $\mathscr{B}$. If we assume that $\mathscr{B}$ is primitive, that is, no two members of $\mathscr{B}$ divide one another, then $\cf=\cf_{\mathscr{B}}$ and such a $\mathscr{B}$ is unique. Classically, see  \cite{Hal}, Section V, $\S4$, for each {\bf primitive} (which is our standing assumption from now on) and {\bf infinite}  set $\mathscr{B}$, we have the following results:
\beq\label{erd1} 0=\underline{d}(\mathscr{B})\leq \overline{d}(\mathscr{B})\leq\frac12,\eeq where
$\underline{d},\overline{d}$ stand for the lower and upper density, respectively (for the density notions, see Section~\ref{s:densities}), while
\beq\label{erd2} \delta(\mathscr{B})=0,\eeq
where $\delta$ stands for the logarithmic density. Moreover,
\beq\label{erd3}
\sum_{b\in\mathscr{B}}\frac1{b\log b}<+\infty.\eeq

Following \cite{Ha}, if $\cf=\cf_{\mathscr{B}}$ satisfies~\eqref{behr2} then $\mathscr{B}$ is called a {\em Behrend} set (and the corresponding $\mathscr{B}$-free subshifts will be called {\em Behrend subshifts}). Condition~\eqref{behr2} indicates hence that we intend to study dynamics of some sparsed sets. In fact, it is not only density as in~\eqref{behr2} or the logarithmic density as in~\eqref{erd2} that vanish. Actually,  when $\mathscr{B}$ is Behrend,  the upper Banach density $BD^\ast(\cf_{\mathscr{B}})$ of $\cf_{\mathscr{B}}$ equals~0 \cite{Dy-Ka-Ku-Le}, and also $BD^\ast(\mathscr{B})=0$, see Corollary~\ref{c:BerZero} in Section~\ref{s:mincompl}. Such sparsed sets are interesting from the point of view of so called non-conventional ergodic theorems (along subsequences) first pointed out in  \cite{Pa}, see also more recent articles \cite{Fr-Ka-Le}, \cite{Ma-Ri}, \cite{Mu}.

However, there are at least three more direct reasons which make Behrend $\mathscr{B}$-free subshifts of special interest in dynamics.  Indeed,
in general, the complement of $\cf=\cf_{\mathscr{B}}$ is a {\em set of multiples}, $\Z\setminus \cf=\cm_{\mathscr{B}}:=\cup_{b\in\mathscr{B}}b\Z$.
So firstly,~\eqref{behr2} means that the  subshift
$
(X_\eta,S)$ with $\eta=\raz_{\cf_{\mathscr{B}}}$, called a $\mathscr{B}$-{\em free subshift}, captures information about a ``typical'' natural number (since the density of $\cm_{\mathscr{B}}$ is now~1). If we can obtain interesting dynamical results, we can count on proving interesting facts in number theory, see also \cite{Ha} for the number-theoretic point of view on interest in Behrend sets.\footnote{According to \cite{Ha}, Section 1.3, p. 36, Erd\"os in 1979 wrote ``It seems very difficult to obtain a necessary and sufficient condition that if $a_1<a_2<\ldots$ is a sequence of integers then almost all integers $n$ should be a multiple of one of the $a_i$'s''. From the dynamical point of view we have no problem to characterize the Berend subshifts: these are precisely those $\mathscr{B}$-free subshifts which are proximal and have zero entropy, see e.g.\ \cite{DoMON} for the definition of entropy.}    Secondly, it is not hard to see that if $\mathscr{B}$ is a Behrend set, then the all zero sequence $0^{\Z}$ belongs to $X_\eta$ and it is known that the Behrend subshifts have only one invariant measure,\footnote{This is due to the fact that once \eqref{behr2} holds then, by~\eqref{behr1},  the upper Banach density of $\cf$ is also zero.} the Dirac measure $\delta_{0^{\Z}}$ \cite{Dy-Ka-Ku-Le}. That is, from the dynamical point of view, these are uniquely ergodic models of the one-point system (in particular, topological entropy of such systems is zero). This makes them trivial from the ergodic theory point of view but we will see that the dynamics of  Behrend subshifts are rich and complex from the topological dynamics point of view. Finally, Behrend subshifts seem to be crucial to understand the general theory of $\mathscr{B}$-free subshifts in the proximal case (see Footnote~\ref{f:proximal} for the definition of proximality).  As a matter of fact, even though the sets of multiples have been studied in number theory for about 100 years, dynamically they have been investigated for the first time in the celebrated Sarnak's article \cite{Sa} concerning (among other problems) the square-free system\footnote{Square-free system is a special instance of a $\mathscr{B}$-free system with $\mathscr{B}$ {\em coprime}, i.e.\ $\mathscr{B}$ consists of elements which are pairwise coprime. It is not hard to see that in the coprime case we obtain a Behrend subshift if and only if $\sum_{b\in\mathscr{B}}1/b=\infty$. Furthermore, the latter (in the coprime case) is equivalent to zero entropy.} given by $\mathscr{B}=\{p^2:\: p\in\PP\}$ ($\PP$ stands for the set of primes).  This system, similarly to the Behrend case, is proximal but obviously it does not satisfy~\eqref{behr2}. The latter implies that the corresponding $\mathscr{B}$-free system  has positive entropy and is taut (see Section~\ref{s:wprowadzenie} for this notion). It follows that the square-free system has plenty of interesting invariant measures which makes possible an analysis using ergodic theory tools - for the general theory of $\mathscr{B}$-free systems see \cite{Ab-Le-Ru}, \cite{Dy-Ka-Ku-Le} and \cite{Ku-Le-We}. Behrend subshifts are proximal and among proximal $\mathscr{B}$-free systems they are characterized by their zero entropy property (cf.\ Corollary~\ref{c:charE}). Building on some recent progress of the theory of $\mathscr{B}$-free subshifts in \cite{Ke} and \cite{Ku-Le(jr.)} and using \cite{Dy-Ka-Ku-Le}, it is noticed in \cite{Le-Ri-Se} that for each proximal $\mathscr{B}$-free system there exists a (unique) taut  $\mathscr{B}'$-free subsystem such that $\cf_{\mathscr{B}'}\subset \cf_{\mathscr{B}}$ and the density of the difference of these two sets vanishes, which ``justifies'' the conclusion that $(X_\eta,S)$ is ``relatively Behrend'' over $(X_{\eta'},S)$ (in fact, the hereditary closures\footnote{\label{f:herclo} Given a subshift $(X_y,S)$ with $X_y\subset\{0,1\}^{\Z}$ by its hereditary closure we mean the smallest hereditary (cf.\ Footnote~\ref{f:hereddef}) subshift $\widetilde{X}_y$ containing $X_y$.} of these two systems have the same sets of invariant measures). We should add that dynamics in the taut case is much better understood since it leads to the theory of hereditary\footnote{\label{f:hereddef} A subshift $(X,S)$ with $X\subset\{0,1\}^{\Z}$ is hereditary if whenever $x\in X$ and $y\leq x$ (coordinatewise) then $y\in X$.} subshifts. Hence, to understand the ``relative Behrend'' case, whence the general proximal case, it seems to be reasonable first to understand possible dynamics of the Behrend subshifts themselves.

In what follows, there will be no special referring to illustrate an importance of the three above reasons and the paper is simply focused on showing how rich the dynamics of Behrend sets can be and what kind of consequences we can derive from it.

\subsection{Toward results}
Together with $\mathscr{B}$-free subshifts, we will also consider subshifts $(X_{\mathscr{B}},S)$, called $\mathscr{B}$-{\em admissible}, where  $X_{\mathscr{B}}$ consists of all $0-1$-sequences whose support misses at least one residue class modulo any $b\in\mathscr{B}$. As $\cf_{\mathscr{B}}$ misses the zero residue class mod all $b\in\mathscr{B}$,
\beq\label{g1}X_\eta\subset X_{\mathscr{B}},\eeq
the latter subshift being obviously hereditary.

Following \cite{Meyer},  the (symmetrized) set of prime numbers with 1 added to it can be seen as a $\mathscr{B}$-free set (take $\mathscr{B}=\{pq:\: p,q\in\PP\}$)\footnote{If we take $\mathscr{B}=\{pq:\; p\neq q\in\PP\}$ then $\cf_{\mathscr{B}}$ equals the (symmetrized) support of the von Mangoldt function $\Lambda$. Other classical sets can be obtained similarly by considering the set of almost $k$-prime numbers, i.e.\
$\mathscr{B}=\PP_k:=\{p_1,\ldots,p_k: p_i\text{ are primes for }i=1,\ldots,k \}$, see also Section~\ref{s:przykladziki}.}  with $\mathscr{B}$ clearly Behrend. So, remembering our convention contained in Footnote~\ref{f:subshift}, the subshift $(X_{\raz_{\PP\cup\{1\}}},S)$ is Behrend. It is now rather standard to notice that this subshift is conjugated to the {\em subshift of primes numbers} $(X_{\raz_{\PP}},S)$,\footnote{If $y,z\in\{0,1\}^{\Z}$ are isolated points in $X_y$ and $X_z$, respectively, and $y,z$ are asymptotic, that is, they differ only on finitely many coordinates, then the subshifts $(X_y,S)$ and $(X_z,S)$ are conjugated. Indeed, the map $S^iy\mapsto S^iz$ ($i\in\Z$) is uniformly continuous
(given $\vep>0$, select $N>0$ so that $d(S^iy,S^iz)<\vep/3$ for all $|i|>N$; then choose $0<\delta<\vep/3$ so that in the $\delta$-neighbourhood of $S^ky$, $k=-N,\ldots,N$, there are no points of the form $S^\ell y$ with $\ell\neq k$. It follows that if $d(S^iy,S^jy)<\delta$ then $d(S^iz,S^jz)<\vep$).}
so in what follows we consider this more natural subshift.

One can ask how big is this system, and as noticed\footnote{This observation is due to A. Schinzel.} in \cite{Dy-Ka-Ku-Le}, even the fact that this subshift is uncountable depends on the validity of an ``unprovable'' old conjecture, namely in this case, of Dickson's conjecture. Recalling that the entropy of the corresponding subshift is zero, one can ask what is the complexity\footnote{\label{f:complexity} Given a subshift $(X,S)$, we say that a block appears in $X$ if there is $y\in X$ such that the block appears in $y$. By the complexity of $(X,S)$ we mean the function $n\mapsto {\rm cpx}_{X}(n)$, where ${\rm cpx}_X(n)$ stands for the number of blocks of length $n$ appearing in $X$. When $X=X_y$, then we write ${\rm cpx}_y(n)$ for ${\rm cpx}_{X_y}(n)$. If the entropy is zero, the growth of the complexity function must be sub-exponential, cf.\ Footnote~\ref{f:entr}.} of the prime numbers subshift. As T.\ Tao \cite{Tao} noticed to us, there is an upper bound for the complexity of $\raz_{\PP}$ and, assuming the $k$-tuple Hardy-Littlewood conjecture\footnote{The Hardy-Littlewood conjecture could be replaced by Dickson's conjecture in Theorem \ref{t:complexityP}.}, also a  lower bound. Throughout, we denote by $\log$ the natural logarithm.
\begin{Th}\label{t:complexityP}  We have
\beq\label{tao}
{\rm cpx}_{\raz_{\PP}}(n)\ll (4+o(1))^{n/\log n}
\eeq
and
\beq\label{tao1}
(2+o(1))^{n/\log n}\ll {\rm cpx}_{X_{\PP}}(n)\ll (4+o(1))^{n/\log n}
\eeq
when $n\to\infty$.
If the Hardy-Littlewood conjecture is true then
\beq\label{tao2}
(2+o(1))^{n/\log n}\ll {\rm cpx}_{\raz_{\PP}}(n)
\eeq
when $n\to\infty$.
\end{Th}
We will provide Tao's proof (private correspondence) of the above theorem in Section~\ref{s:TTao}.

By proving $X_{\raz_{\PP}}\subset X_{\raz_{\PP_2}}$ (Proposition \ref{prop:Dickson's_and_semiprimes}) and using Theorem~\ref{t:complexityP} we obtain a lower bound $(2+o(1))^{n/\log n}\ll {\rm cpx}_{\raz_{\PP_2}}(n)$ for the subshift $(X_{\raz_{\PP_2}},S)$ of semi-primes, conditionally on the Hardy-Littlewood conjecture.
 Moreover, by deriving some consequences of Dickson's conjecture, we  prove in Sections~\ref{s:DC2} and~\ref{s:DC3} some other estimates, namely:

\begin{Th}\label{t:SP}  We have
\beq\label{pedwa}
{\rm cpx}_{\raz_{\PP_2}}(n)\ll (4+o(1))^{\frac{23n\log^2(\log(n))}{\log(4)\log(n)}}.
\eeq
If Dickson's conjecture is true then
\beq\label{tao3}
(2+o(1))^{n\log\log n/2\log n}\ll {\rm cpx}_{\raz_{\PP_2}}(n)
\eeq
when $n\to\infty$.
\end{Th}

A natural question arises how fast can grow the complexity in a general Behrend $\mathscr{B}$-free subshift. We will show that any sub-exponential growth rate of this function can be realized by a Behrend $\mathscr{B}$-free subshift which is even $\mathscr{B}$-admissible.

\begin{Th}\label{t:BehrendA}
Let $\rho:\N\rightarrow\N$ be a function such that $\frac{\rho(n)}{n}\searrow 0$ as $n\rightarrow +\infty$.
There exists a Behrend set $\mathscr{B} \subset\PP$ such that $X_{\eta}=X_{\mathscr{B}}$.  Moreover, $\limsup_{n\to\infty} \frac{{\rm cpx}_{\eta}(n)}{2^{\rho(n)}}\ge 1$.
\end{Th}

As one of our motivations is to prove the uncountability of certain subshifts, one could hope that if the complexity grows ``almost'' exponentially fast, then the subshift must be uncountable (if the entropy of a subshift is positive, then it must be uncountable) but  Cyr and Kra \cite{Cy-Kr} constructed countable subshifts with arbitrarily fast sub-exponential complexity. Their construction is however based on a richness of periodic points in some countable systems. This is not our case, as a general $\mathscr{B}$-free subshift has only one minimal (cf. Footnote~\ref{f:minimaldef}) subsystem \cite{Dy-Ka-Ku-Le}, and therefore we deal with (special) uniquely ergodic subshifts whose cardinality is not clear.  One more attempt to  prove uncountability (of a Behrend $\mathscr{B}$-free subshift) could be based on the equality in the inclusion~\eqref{g1} because, as noticed in \cite{Dy-Ka-Ku-Le}, $\mathscr{B}$-admissible subshifts\footnote{As shown in \cite{Dy-Ka-Ku-Le}, Behrend $\mathscr{B}$-admissible subshifts have still zero entropy. In fact, $\delta_{0^{\Z}}$ is the only invariant measure.}  are always uncountable. This approach works for Behrend sets considered in Theorem~\ref{t:BehrendA},  but for $\mathscr{B}=\{pq:\:p,q\in\PP\}$ we fail again, because whereas the subshift $(X_{\raz_{\PP\cup\{1\}}},S)$ has isolated points,\footnote{\label{f:izol} Given a subshift $(X,S)$, by ${\rm Isol}(X)$ we denote its set of isolated points. This set is open, countable and $S$-invariant. Hence,
$X_0:=X\setminus{\rm Isol}(X)$
defines a subshift $(X_0,S)$ which is a subsystem of $(X,S)$. If we consider a particular situation when $X=X_u$, then either $u$ is not an isolated point and then $(X_u)_0=X_u$ or $u$ is isolated (then so are all points from its orbit), and then
$
(X_u)_0=X_u\setminus\{S^nu:\:n\in\Z\}$.
Note also that $u$ is an isolated point if and only if there is a block $B$ that appears in $u$ only finitely many times.
Note that for the prime numbers subshift there are isolated points because  configurations like $\{2,3\}$ or $\{3,5,7\}$ appearing in $\PP$ are not $\PP$-admissible and  can appear in $\eta$ only finitely many times.
Now, Dickson's conjecture can be formulated as: $(X_{\raz_{\PP}})_0=X_{\PP}$ and the present theorem can be viewed as an instance of validity of Dickson's conjecture for a Behrend set.} $\mathscr{B}$-admissible subshifts have no isolated points (see Section~\ref{s:isolated} for details). It is also worth to notice that $\mathscr{B}$-admissible subshifts have {\bf always} ``large'' complexities, namely, it is at least the left-hand side of the inequality~\eqref{tao1} in Theorem~\ref{t:complexityP}:
\begin{proposition}\label{p:staszek1}
For each $\mathscr{B}$, ${\rm cpx}_{X_{\mathscr{B}}}(n)\geq (2+o(1))^{\frac{n}{\log n}}$.\end{proposition}
Returning to the  case $X_\eta=X_{\mathscr{B}}$,
 note that the set $\{0,2\}$ is {\bf always} $\mathscr{B}$-admissible, so once we have $X_\eta=X_{\mathscr{B}}$ then (since now $X_\eta$ has no isolated points) the corresponding set of $\mathscr{B}$-free numbers will contain infinitely many pairs $(t,t+2)$ of {\bf twin} $\mathscr{B}$-free numbers.\footnote{
For sets $\mathscr{B}\subseteq\PP$ such that $\PP\setminus \mathscr{B}$ has at most 2 elements,  the result of Bennett \cite{Ben} on Pillai's type equation $a^x-b^y=c$ implies that that there are at most two pairs of twin $\mathscr{B}$-free numbers.}

Since the $\mathscr{B}$-admissible subshifts seem to be especially interesting from the number theory point of view, it is natural to ask whether these subshifts can be transitive, i.e.\ have points whose orbits are dense\footnote{The transitivity of a topological system $(X,T)$ is equivalent to the fact that for each nonempty open sets $U,V\subset X$, there is $n\in\Z$ such that $U\cap T^nV\neq\emptyset$.}. A Chinese Remainder Theorem tells us that this is the case when $\mathscr{B}$ is coprime but a kind of surprise is that this is the {\bf only} possibility for the existence of a ``dense'' admissible configuration:

\begin{Th}\label{t:MS} For each $\mathscr{B}$, the subshift $(X_{\mathscr{B}},S)$ is transitive if and only if $\mathscr{B}$ is coprime.\end{Th}

This theorem has a number of consequences. As we have already noticed the most prominent example of a (proximal) $\mathscr{B}$-free system is the square-free system \cite{Sa}. The natural class beyond it is the class of so called Erd\"os $\mathscr{B}$-free systems - here $\mathscr{B}$ is said to satisfy the Erd\"os condition if it is infinite, coprime and $\sum_{b\in\mathscr{B}}1/b<\infty$.  While this seems to be just a technical condition, Theorem~\ref{t:MS}, \cite{Ke}, \cite{Dy-Ka-Ku-Le} \textcolor{black}{yield} a dynamical characterization\footnote{\label{f:entr}Given a dynamical system $(X,T)$, by $h(X,T)$ we denote its topological entropy. We recall that for a subshift $(X_y,S)$, we have $h(X_y,S)=\lim_{n\to\infty}\frac1n\log {\rm cpx}_y(n)$, cf.\ Footnote~\ref{f:complexity}.} of the Erd\"os case:

\begin{corollary} \label{c:charE}(i) $\mathscr{B}$ is Erd\"os if and only if $X_\eta=X_{\mathscr{B}}$ and $h(X_\eta,S)>0$.

(ii) $\mathscr{B}$ is Behrend if and only if $(X_\eta,S)$ is proximal and $h(X_\eta,S)=0$.\end{corollary}

We recall that also (hereditary closures, cf.\ Footnote~\ref{f:herclo}) subshifts $(\widetilde{X}_\eta,S)$  are quite often transitive (see \cite{Dy-Ka-Ku-Le}) - more precisely, the transitivity is equivalent to the fact that if a block $B$ appears on $\eta$ then a certain bigger (coordinatewise) block than $B$ reappears in $\eta$ infinitely often. In particular, if $\eta$ is recurrent then $(\widetilde{X}_\eta,S)$ is transitive. An important generalization of Erd\"os property of $\mathscr{B}$, complementary to the concept of Behrend set, is tautness \cite{Ha} (see Section~\ref{s:wprowadzenie} for more details).
We obtain that for each taut $\mathscr{B}$, the subshift $(\widetilde{X}_\eta,S)$ is transitive (this follows from the recurrence of $\eta$ which is a consequence of Theorem~\ref{t:nosnik} below and the fact that $\eta$  is quasi-generic for the Mirsky measure \cite{Dy-Ka-Ku-Le}).

\begin{corollary} \label{c:charEE} Assume that $\mathscr{B}$ is taut and infinite. If $\mathscr{B}$ is not Erd\"os then $\widetilde{X}_\eta\subsetneq X_{\mathscr{B}}$.\end{corollary}

\begin{corollary}\label{c:charE1} Assume that  $\mathscr{B}$ is taut and we have $X_\eta=X_{\mathscr{C}}$ for {\bf some} $\mathscr{C}$. Then $\mathscr{B}$ is  Erd\"os and $\mathscr{B}=\mathscr{C}$. \end{corollary}

\begin{corollary} \label{c:charEEE} Assume that $\mathscr{B}$ is infinite and $(X_\eta,S)$ is minimal.\footnote{\label{f:minimaldef}Minimal systems are those in which each orbit is dense.}  Then $X_\eta\subsetneq \widetilde{X}_\eta\subsetneq X_{\mathscr{B}}$.\end{corollary}




\section{Necessary facts from the theory of $\mathscr{B}$-free systems}\label{s:wprowadzenie} We recall that throughout we assume that $\mathscr{B}\subset\N$ is primitive and that $1\notin\mathscr{B}$ (unless it is stated otherwise).
\subsection{Densities} \label{s:densities} Given a subset $A\subset\N$, its {\em lower density} is defined as
$$
\underline{d}(A)=\liminf_{n\to\infty}\frac1n|A\cap[1,n]|.$$
When $\liminf$ is replaced by $\limsup$, we speak about the {\em upper density} $\ov{d}(A)$. If the lower density equals the upper density of $A$ then we say that $A$ {\em has density} $d(A)=\underline{d}(A)=\ov{d}(A)$. Similarly, we speak about logarithmic densities, in particular, $A$ {\em has logarithmic density} if the limit
$$
\delta(A):=\lim_{n\to\infty}\frac1{\log n}\sum_{A\ni j\leq n}\frac1j
$$
exists. Finally, we define the {\em upper Banach density} of $A$ as
$$
BD^\ast(A):=\limsup_{n\to\infty}\frac1n\max_{m\in\N} |A\cap[m,m+n]|.$$
If $A$ has density then it has logarithmic density and
$$
\delta(A)\leq d(A)\leq BD^\ast(A).$$

\subsection{Introduction to the theory of $\mathscr{B}$-free subshifts}
We set $\eta=\eta_{\mathscr{B}}:=\raz_{\cf_{\mathscr{B}}}\in\{0,1\}^{\Z}$ and $X_\eta:=\ov{\{S^n\eta:\:n\in\Z\}}$, where $S$ is the left shift on $\{0,1\}^{\Z}$,  to define $(X_\eta,S)$ a $\mathscr{B}$-free subshift. By their very definition $\mathscr{B}$-free subshifts are transitive.  The dynamics of these subshifts vary in a significant way depending on the arithmetic properties of $\mathscr{B}$: indeed, it varies\footnote{As proved in \cite{Dy-Ka-Ku-Le}, each $\mathscr{B}$-free system has a unique minimal subset. Proximality corresponds to the smallest possible minimal subset (a fixed point), while minimality corresponds to the largest possible minimal subset.}  from proximality\footnote{\label{f:proximal}A topological dynamical system $(X,T)$ is {\em proximal} if for every pair $x,y\in X$ there is a sequence $(q_n)$ such that $d(T^{q_n}x,T^{q_n}y)\to0$. Such a system has necessarily a fixed point which is the unique minimal subset.} to minimality. Both these dynamical properties have arithmetic characterizations: the proximality of $(X_\eta,S)$ is equivalent to the fact that $\mathscr{B}$ contains an infinite coprime subset \cite{Dy-Ka-Ku-Le} (it is also equivalent to the fact that the all zero sequence belongs to $X_\eta$ \cite{Dy-Ka-Ku-Le}), while the minimality is equivalent to the fact that $\mathscr{B}$ does not contain a rescaled copy of an infinite  coprime set \cite{Ka-Ke-Le}. Moreover, $(X_{\eta},S)$ is minimal if and only if it is a Toeplitz system, in fact, it is equivalent to $\eta$ itself being a Toeplitz sequence\footnote{I.e.\ for each $n\in\Z$ there is $k_n\in\N$ such that $\eta(n)=\eta(n+jk_n)$ for each $j\in\Z$. See e.g.\ \cite{Do} for the theory of Toeplitz systems.} \cite{Ke1}.

By the Davenport-Erd\"os theorem (see e.g.\ Thm.~0.2 in \cite{Ha}), the set $\cf_{\mathscr{B}}$ has logarithmic density  which is equal to its upper density and
\beq\label{Dav-Erd}
\delta(\cf_{\mathscr{B}})=\lim_{M\to\infty} d(\cf_{\{b\in\mathscr{B}:\:b\leq M\}}).\eeq
Moreover, if $(N_k)$ is any sequence ``realizing'' the upper density:
$$
\lim_{k\to\infty}\frac1{N_k}\big| \cf_{\mathscr{B}}\cap[1,N_k]\big|=
\ov{d}(\cf_{\mathscr{B}}),$$
then $\eta$ is generic along $(N_k)$ for the so-called Mirsky measure $\nu_\eta$, see \cite{Dy-Ka-Ku-Le},  which is an invariant measure for the subshift $(X_\eta,S)$.
In both classes: proximal and minimal, the entropy can be positive and also zero. In fact, using \cite{Ku-Le(jr.)}, \cite{Ke}, it has been noticed in \cite{Le-Ri-Se} that:

\begin{Th} [\cite{Le-Ri-Se}]\label{t:entropyprox} If $(X_\eta,S)$ is proximal, then $h(X_\eta,S)=\nu_\eta(C_{\{0\},\emptyset})\log2$, where $C_{\{0\},\emptyset}:=\{y\in X_\eta:\:y(0)=1\}$.\end{Th}

From the arithmetic point of view, classically, there were studied subclasses of $\mathscr{B}$-free subsets. The most ``sparse'' sets are those coming from {\em Behrend} sets. Namely, $\mathscr{B}$ is Behrend if the logarithmic density of $\cf_{\mathscr{B}}$ equals zero. Equivalently, the natural density of such sets is zero, that is, $\nu_\eta(C_{\{0\},\emptyset})=0$. Moreover (see e.g.\ Cor.~0.14 in \cite{Ha}),
\beq\label{BehIne}
\mbox{if $\mathscr{B}=\mathscr{B}_1\cup\mathscr{B}_2$ is Behrend then either $\mathscr{B}_1$ or $\mathscr{B}_2$ is Behrend.}
\eeq

As the all zero sequence $0^{\Z}$ is in $X_\eta$, the corresponding Behrend $\mathscr{B}$-free subshift $(X_\eta,S)$ is proximal. Hence, each Behrend set contains an infinite coprime \textcolor{black}{subset} and, in view of Theorem~\ref{t:entropyprox}, their entropy is zero. In fact, as shown in \cite{Dy-Ka-Ku-Le}, $\nu_\eta$ is just the Dirac measure $\delta_{0^{\Z}}$ at the fixed point, and it is the only $S$-invariant measure on $X_\eta$. Remaining in the proximal case, on the other extreme, we have  $\mathscr{B}$ which are called
{\em Erd\"os}, i.e.\ sets which are infinite, coprime and {\em thin}, i.e., $\sum_{b\in\mathscr{B}}\frac1b<+\infty$. As for coprime sets $\mathscr{B}$, we have
\beq\label{wzorek}
d(\cf_{\mathscr{B}})=\nu_\eta(C_{\{0\},\emptyset})=\prod_{b\in \mathscr{B}}\left(1-\frac1b\right),\eeq \cite{Ab-Le-Ru} (see also the earlier article \cite{Sa}), in the Erd\"os case, the entropy is positive by Theorem~\ref{t:entropyprox}. A prominent example in this class is the square-free system for which $\mathscr{B}=\{p^2:\: p\in\PP\}$ studied first by Sarnak \cite{Sa}. Also, in the minimal case, the entropy can be both zero \cite{Dy-Ka-Ku-Le} and positive \cite{Ke2}.

Another classical arithmetic notion in the context of $\mathscr{B}$-free sets is that of tautness. Namely, $\mathscr{B}$ is {\em taut} if for each $b\in\mathscr{B}$ the logarithmic density of $\cf_{\mathscr{B}}$ is strictly smaller than the logarithmic density of $\cf_{\mathscr{B}\setminus\{b\}}$. Each {\em thin} set $\mathscr{B}$ is taut \cite{Dy-Ka-Ku-Le}. Behrend sets are {\bf not} taut,\footnote{Since no one-element set is Behrend, by~\eqref{BehIne}, Behrend sets are not taut.  In fact, tautness is characterized by the fact that $\mathscr{B}$ (which is always assumed to be primitive) does not contain a rescaled copy of a Behrend set, see \cite{Ha}.} so in particular $\sum_{b\in\mathscr{B}}\frac1b=+\infty$ for Behrend sets, whereas Erd\"os sets are. Also, as noticed by A.\ Dymek (see \cite{Ka-Ke-Le}), all $\mathscr{B}$ which yield minimal $\mathscr{B}$-free systems are necessarily taut. We will use the following dynamical characterisation of tautness proved in \cite{Ke}, \cite{Ku-Le(jr.)}:

\begin{Th} [\cite{Ke}, \cite{Ku-Le(jr.)}] \label{t:nosnik} $\mathscr{B}$ is taut if and only if ${\rm supp}\,\nu_\eta=X_\eta$.\end{Th}

\subsection{$\mathscr{B}$-free systems as model set systems}
Assume that $\mathscr{B}=\{b_1,b_2,\ldots\}$. Let
$$
H:=\ov{\{(n,n,\ldots):\:n\in\Z\}}\subset\prod_{k=1}^\infty\Z/b_k\Z.$$
Then $H$ is a compact Abelian group with Haar measure $m_H$.
Consider also the translation $T(h_1,h_2,\ldots)=(h_1+1,h_2+1,\ldots)$ on $H$.
This translation is ergodic with respect to $m_H$.
Denote
$$
W:=\{h\in H:\:h_k\neq0\text{ for each }k\in\Z\}.$$
The {\em window} $W$ gives us a natural way of coding points $h\in H$ according to the visits of the consecutive elements $T^nh$ either to $W$ or $W^c$. Formally, let
$\va:H\to\{0,1\}^{\Z}$ be defined by
$$
\va(h)(n):=\left\{
\begin{array}{ccc} 0&\text{if}& (\exists k\geq1)\;b_k|n+h_k\\
                   1&\text{otherwise.}&\end{array}\right.$$
It is then not hard to see that $\eta=\va(0,0,\ldots)$.
Following \cite{Dy-Ka-Ku-Le}, we obtain:
\begin{itemize}
 \item The Mirsky measure $\nu_\eta$ is equal to $\va_\ast(m_H)$ the image of $m_H$ via $\varphi$.
\item  $\mathscr{B}$ is Behrend if and only if $m_H(W)=0$.\end{itemize}
We also recall that tautness can be characterized in terms of the properties of the window. Namely, the following has been proved in \cite{Ka-Ke-Le}:
\begin{itemize}
\item  $\mathscr{B}$ is taut if and only if the window W is Haar regular, i.e. ${\rm supp}(m_H |W )=W$.\end{itemize}

\subsection{$\mathscr{B}$-admissible subshifts}
Yet, to each $\mathscr{B}$, we can associate another natural subshift $(X_{\mathscr{B}},S)$ called $\mathscr{B}$-{\em admissible} subshift, as
$$X_{\mathscr{B}}:=\{x\in\{0,1\}^{\Z}:\:{\rm supp}\,x\text{ is }\mathscr{B}-\text{admissible}\},$$
where ${\rm supp}\,x:=\{i\in\Z:x_i=1\}$ is the support of $x$.
We recall that $A\subset \Z$ is $\mathscr{B}$-admissible, if for each $b\in\mathscr{B}$, $|A~{\rm mod}~b|<|A|$. By the very definition, 
$b\Z$ is disjoint \textcolor{black}{from} $\cf_{\mathscr{B}}$ for every $b\in\mathscr{B}$, whence $X_\eta\subset X_{\mathscr{B}}$.

It is easy to see that $X_{\mathscr{B}}$ is {\em hereditary}, i.e. if $x\in X_{\mathscr{B}}$ and $y\leq x$ (coordinatewise) then $y\in X_{\mathscr{B}}$. Therefore,
\beq\label{zawie}
X_\eta\subset \widetilde{X}_\eta\subset X_{\mathscr{B}},\eeq
where $\widetilde{X}_\eta$ is the hereditary closure of $X_\eta$, i.e. the smallest hereditary subshift containing $X_\eta$. As proved in \cite{Ab-Le-Ru}, in the Erd\"os case, $X_\eta=X_{\mathscr{B}}$ (in fact, this equality was first proved by Sarnak \cite{Sa} in the square-free case).

\begin{Remark}\label{r:Bskonczony} In general, the three subshifts in \eqref{zawie} are different. However, if $\mathscr{B}$ is finite and coprime then $\widetilde{X}_\eta= X_{\mathscr{B}}$. Indeed, this follows for example from the equivalence of (ii) and (iii) in Proposition~2.5 \cite{Ab-Le-Ru} which holds in the finite, coprime case.\footnote{A simple argument was pointed out by J.\ Ku\l aga-Przymus: For each $n\in\Z$, $\eta(m)=0$ if and only if $m\in  \bigcup_{b\in\mathscr{B}}(b\Z-n)=\bigcup_{b\in\mathscr{B}}(b\Z-(n~{\rm mod}\; b))$; since, by the Chinese Remainder Theorem, $(n~{\rm mod}\;b)_{b\in\mathscr{B}}$ realizes any configuration  $(r_b)_{b\in\mathscr{B}}$ of residue classes mod~$b$, $b\in\mathscr{B}$, the claim follows.}   \end{Remark}

The following observation tells us that in the family of {\bf all} $\mathscr{B}$--admissible subshifts there exists the smallest element.

\begin{proposition}\label{p:staszek}
For each $\mathscr{B}$, we have $X_{\PP}\subset X_\mathscr{B}$.
\end{proposition}
\begin{proof} Suppose that $A\subset\Z$ is $\PP$-admissible but is not $\mathscr{B}$-admissible. It follows that, for some $b\in\mathscr{B}$ and for each $r=0,1,\ldots,b-1$ there exists $j_r\in\Z$ such that $r+j_rb\in A$. Take any $p|b$ and note that $\{r+j_rb:\:r=0,1,\ldots, b-1\}$ contains all residue classes mod~$p$, a contradiction.
\end{proof}

\begin{Remark} In fact, the above argument uses only the following elementary observations:
 if, for every $b\in\mathscr{B}$, we choose $1<c_b|b$, then $X_{\{c_b:\:b\in\mathscr{B}\}}\subset X_{\mathscr{B}}$.\end{Remark}

\subsection{Examples of Behrend sets} \label{s:przykladziki} A natural source of examples of Behrend sets are non-thin coprime sets, namely (cf.\ \eqref{wzorek}):
\beq\label{behre1}
\mbox{A coprime set $\mathscr{B}$ is Behrend if and only if $\sum_{b\in\mathscr{B}}\frac1b=+\infty$.}\eeq
In particular, all subsets $\mathscr{P}\subset\PP$ satisfying $\sum_{p\in\mathscr{P}}\frac1p=+\infty$ are Behrend.

To see other examples, recall that the function $\Omega:\N\to\N$ counts, given $n\in\N$, the number of prime divisors (with multiplicity) of $n$. For $k=0,1,2,\ldots$, denote by $\PP_k$ the set of {\em almost $k$-prime numbers},~\footnote{We recall that the number of $k$-almost prime numbers less than $n$ is  $(1+o_k(1))\frac{n(\log\log n)^{k-1}}{(k-1)!\log n}$,  see (\ref{NT:pnt61}).} that is,
$$
\PP_k:=\{n\in\Z:\: \Omega(n)=k\}.$$
Hence $\PP_0=\{1,-1\}$, $\PP_1=\PP\cup(-\PP)$, i.e. the set of primes, $\PP_2$ is the set of semi-primes, etc.
It is not hard to see that, for each $k\geq1$, we have
\beq\label{niezle}
   \cF_{\PP_k}=\bigcup_{\ell<k}\PP_\ell.\eeq
  All the sets $\bigcup_{\ell<k}\PP_\ell$ have zero density, so all the sets $\PP_k$, $k\geq1$, are Behrend sets.
In particular,
$$
X_{\raz_{\bigcup_{\ell<k}\PP_\ell}}\subset X_{\PP_k}.$$
Moreover, the non-isolated points of $X_{\raz_{\PP_k}}$ belong to $X_{\PP_k}$, see Remark \ref{r:isoladm}.

For example, the  semi-primes-admissible subshift contains the subshift of  primes, i.e.\ $X_{\raz_{\PP\cup(-\PP)}}$.

The above ``tower'' of Behrend $\mathscr{B}$-free systems has a natural generalization coming from the following observation.

\begin{proposition}\label{p:marsta} The set
$\{bc:\: b,c\in \mathscr{B}, b\neq c\}$~\footnote{Note that this set is primitive if $\mathscr{B}$ is coprime; in general, we consider the primitive basis of this set.}  is Behrend, whenever $\mathscr{B}$ is Behrend.\end{proposition}
\begin{proof} If $F\subset\N$ then by ${\rm spec}(F)$ we denote the set of all prime divisors of elements from $F$.

To prove our claim, first note that if $\mathscr{B}$ is Behrend, then for each finite set $Q\subset\PP$, the set
$$\mathscr{B}':=\{b\in \mathscr{B}:\: {\rm spec}(b)\cap Q=\emptyset\}$$ is also Behrend. Indeed,
we have $\mathscr{B}=\mathscr{B}'\cup\bigcup_{q\in Q}(\mathscr{B}\cap q\Z)$. Since the union is finite, in view of \eqref{BehIne},
at least one of the sets in the union is Behrend but no one of the sets
$\mathscr{B}\cap q\Z$ is Behrend.

Let us select now a finite set $S\subset \mathscr{B}$ such that $\cM_{S}$ has density close to~1 (we use here \eqref{Dav-Erd}). Let $\mathscr{B}'$ be the set of $b\in \mathscr{B}$ which are coprime to all elements of $S$. Now, $\mathscr{B}'$ is Behrend in view of the first part of the proof.
Choose a finite set
$S'\subset \mathscr{B}'$ so that $\cM_{S'}$ has density close to~1. The elements from $S$ are coprime to the elements of
$S'$, whence the density of $\cM_{S\cdot S'}$ equals to the product of the densities  of
$\cM_{S}$ and $\cM_{S'}$ (see Lemma 4.21 in \cite{Dy-Ka-Ku-Le}), so it is close to~1. As $S\cdot S'\subset\{bc:\:b,c\in\mathscr{B}, b\neq c\}$, the result follows.
\end{proof}

\section{Transitivity of $\mathscr{B}$-admissible subshifts. Proof of Theorem~\ref{t:MS}} Let us first discuss a certain reduction of the problem of transitivity of $(X_{\mathscr{B}},S)$. Given two finite and disjoint sets $A,B\subset\Z$, set
$$
C_{A,B}:=\{x\in X_{\mathscr{B}}: x(n)=1\text{ for each }n\in A, x(n)=0\text{ for each }n\in B\}$$
(we could have considered also the case $A$ and $B$ are not disjoint but clearly $C_{A,B}=\emptyset$ in this case). Note that $C_{A,\emptyset}\neq \emptyset$ if and only if $A$ is $\mathscr{B}$-admissible and furthermore, $C_{A,B}\neq \emptyset$ if and only if $C_{A,\emptyset}\neq \emptyset$ by the definition of $\mathscr{B}$-admissibility.
We have:
$$
S^nC_{A,B}=C_{A+n,B+n},$$
$$
C_{A,B}\cap C_{A',B'}=C_{A\cup A',B\cup B'},$$
so in order to have that this set is non-empty, we must know that  $(A\cup A')\cap (B\cup B')=\emptyset$. Note also that if we aim at showing that $S^nC_{A,B}\cap C_{A',B'}\neq\emptyset$, equivalently that $C_{(A+n)\cup A',(B+n)\cup B'}\neq\emptyset$, we have to show that $C_{(A+n)\cup A',\emptyset}\neq\emptyset$
and that the sets $(A+n)\cup A'$ and $(B+n)\cup B'$ are disjoint, the latter holding if $n$ is large enough. As the transitivity of $(X_{\mathscr{B}},S)$ is equivalent to:   for each finite, disjoint $A,B\subset\Z$, $A',B'\subset\Z$ there exists $n\in\Z$ such that $S^nC_{A,B}\cap C_{A',B'}\neq\emptyset$, using the above observations, we obtain the following:

\begin{lemma}\label{l:trans} $(X_{\mathscr{B}},S)$ is transitive if and only if for each finite, $\mathscr{B}$-admissible sets $A,A'\subset\Z$ there exists $n\in\Z$ such that $(A+n)\cup A'$ is $\mathscr{B}$-admissible.\end{lemma}

\begin{proof} {\em  of sufficiency of Theorem~\ref{t:MS}} \ Fix $A_1,A_2\subset\Z$ finite, $\mathscr{B}$-admissible sets. By Lemma~\ref{l:trans} we need to show that
there is $n\in\Z$ such that
\beq\label{suff1}
\mbox{$(A_1+n)\cup A_2$ is $\mathscr{B}$-admissible.}
\eeq
Let $M:=|A_1|+|A_2|$ and let $b_1,\ldots,b_k$ be all elements of $\mathscr{B}$ which are $\leq M$. We only need to check admissibility with respect to $b_1,\ldots,b_k$. Let $r_{1,j},r_{2,j}$ be a missing residue class mod~$b_j$ in $A_1$ and $A_2$, respectively (perhaps there are more residue class missing, of course). By  the Chinese Remainder Theorem, we can find $n\in\Z$ such that
$$
r_{1,j}+n=r_{2,j}\;{\rm mod}\;b_j,\;j=1,\ldots,k.$$
Now, in the sets $A_1+n$ and $A_2$ the class $r_{2,j}$ is missing mod~$b_j$, whence~\eqref{suff1} holds.\end{proof}

\subsection{Proof of the necessity in Theorem~\ref{t:MS}}
\begin{lemma}\label{omijanie}
Assume that $M\in\N$ and $q_1,\ldots,q_N\in\Z$ for some $N\in\N$. Then, for any natural numbers $c_1,\ldots,c_M>N\cdot M$, there exists $q\in\Z$ such that
$$
\left(\bigcup_{i=1}^M(c_i\Z+q)\right)\cap\{q_1,\ldots,q_N\}=\emptyset.
$$
\end{lemma}

\begin{proof}
The asymptotic density of the set
$$
\bigcup_{j=1}^N\bigcup_{i=1}^M\left(c_i\Z-q_j\right)
$$
does not exceed $
N\sum_{i=1}^M\frac{1}{c_i}<1$
as $c_i>NM$ for every $i$. Now, it is enough to take $q$ outside of this set.
\end{proof}

\begin{lemma}\label{progression}
Let $A$ be a subset of an arithmetic progression $r\Z+t$ with $r>0$ such that if $rk+t\notin A$ for some $k\in\Z$,  then $r(k+1)+t\in A$. Then  there exists $a\in A$ such that $|a|\le r$.
\end{lemma}

\begin{proof}
Clear.
\end{proof}
To complete the proof of Theorem~\ref{t:MS}, it is enough to prove the following result.

\begin{Th} Let $\mathscr{C}\subset \N$ be primitive. Assume that  for any finite $\mathscr{C}$-admissible sets $A_1, A_2\subset \Z$ there exists $m\in\N$ such that $A_1\cup (A_2+m)$ is $\mathscr{C}$-admissible. Then $\mathscr{C}$ is coprime.
\end{Th}

\begin{proof}
Assume that $\mathscr{C}$ is not coprime and let $b\in \mathscr{C}$ be the smallest element such that $\gcd(a,b)>1$ for some $a\in\mathscr{C}$, $b\neq a$. Fix such $a$.
Let
\begin{equation}\label{setT}
T=\{(r,t): r\in\{1,\ldots,\lcm(a,b)\}, t\in\{0,\ldots, r-1\}, \gcd(r,t)\in\cF_{\mathscr{C}\setminus\{a\}}\}.~\footnote{Note that $(r,1)\in T$.}
\end{equation}
We define the number
\begin{equation}\label{numberS}
\Sigma_T=\sum_{(r,t)\in T}\sigma_0(\gcd(r,t)),
\end{equation}
where, given a number $m$,  $\sigma_0(m)$ denotes the number of all divisors  of $m$.

For any $(r,t)\in T$, let $q_{r,t}\in r\Z+t$ be an element such that $q_{r,t}=\gcd(r,t)p_{r,t}$ and $p_{r,t}$ is a prime number not dividing $a$ and such that $p_{r,t}>|T|\cdot \Sigma_T$. The existence of such $q_{r,t}$ is a consequence of Dirichlet Theorem (applied to $\frac{r}{\gcd(r,t)}\Z+\frac t{\gcd(r,t)}$).

Let
\begin{equation}\label{setC}
\mathscr{C}_0=\{c\in\mathscr{C}\setminus\{a\}: c|q_{r,t},\,\text{for some}\, (r,t)\in T\}.
\end{equation}
Any element $c\in\mathscr{C}\setminus\{a\}$ dividing $q_{r,t}$ for some $(r,t)\in T$ is equal to  $lp_{r,t}$ for a  divisor $l$ of $\gcd(r,t)$. Indeed, if $c|q_{r,t}=p_{r,t}\gcd(r,t)$ then, as
$\gcd(r,t)\in\cF_{\mathscr{C}\setminus\{a\}}$, $c$ is not coprime to $p_{r,t}$. Since $p_{r,t}$ is a prime, it follows that $p_{r,t}|c$ and $l:=\frac{c}{p_{r,t}}|\gcd(r,t)$.   Thus
\begin{equation}\label{ileC}
|\mathscr{C}_0|\le \Sigma_T
\end{equation}
and, as $p_{r,t}>|T|\cdot \Sigma_T$,
\begin{equation}\label{jakieC}
c>|T|\cdot \Sigma_T
\end{equation}
for every $c\in\mathscr{C}_0$.

By Lemma~\ref{omijanie} applied to $\{q_1,\ldots,q_N\}=\{q_{r,t}:(r,t)\in T\}$ and $\{c_1,\ldots,c_M\}=\mathscr{C}_0$, thanks to~(\ref{ileC}) and~(\ref{jakieC}),  there exists $q\in\Z$ such that
\begin{equation}\label{omijanie2}
\left(\bigcup_{c\in\mathscr{C}_0}\left(c\Z+q\right)\right)\cap\{q_{r,t}:(r,t)\in T\}=\emptyset.
\end{equation}
Let $n=\max\{q_{r,t}:(r,t)\in T\}$.

We claim that if $r\le\lcm(a,b)$ is a natural number and  $t\in\Z$ then the following implication holds
\begin{equation}\label{r1}
[-n,n]\cap (r\Z+t)\subseteq \bigcup\limits_{c\in\mathscr{C}'\setminus\mathscr{C}_0}c\Z \cup \bigcup\limits_{c\in\mathscr{C}_0}\left(c\Z+q\right)\Rightarrow \gcd(r,t)\in\cM_{\mathscr{C}'},
\end{equation}
for $\mathscr{C}'=\mathscr{C}$ and $\mathscr{C}'=\mathscr{C}\setminus\{a\}$. Indeed,
assume that
\begin{equation}\label{assumption}
[-n,n]\cap (r\Z+t)\subseteq \bigcup\limits_{c\in\mathscr{C}'\setminus\mathscr{C}_0}c\Z \cup \bigcup\limits_{c\in\mathscr{C}_0}\left(c\Z+q\right).
\end{equation}
Without loss of generality we can assume that $t\in\{0,\ldots,r-1\}$.
If $(r,t)\notin T$ then $\gcd(r,t)\in\cM_{\mathscr{C}\setminus\{a\}}\subseteq\cM_{\mathscr{C}'}$. It remains to consider the case $(r,t)\in T$.
Then $q_{r,t}\in [-n,n]\cap (r\Z+t)$ by the choice of $n$. Then  (\ref{omijanie2}) and (\ref{assumption}) yield
$$
q_{r,t}\in \bigcup\limits_{c\in\mathscr{C}'\setminus\mathscr{C}_0}c\Z.
$$
If $\mathscr{C}'=\mathscr{C}\setminus\{a\}$, this leads to a contradiction with the definition of $\mathscr{C}_0$. So we must have $\mathscr{C}'=\mathscr{C}$ and we conclude that $a|q_{r,t}=\gcd(r,t)p_{r,t}$. Since $p_{r,t}$ does not divide $a$, we get $a|\gcd(r,t)$, thus $\gcd(r,t)\in\cM_{\mathscr{C}'}$ for $\mathscr{C}'=\mathscr{C}$. The claim~\eqref{r1} follows.

We set
$$
A_1=[-n,n]\setminus\left(\bigcup\limits_{c\in\mathscr{C}\setminus\mathscr{C}_0}c\Z \cup \bigcup\limits_{c\in\mathscr{C}_0}(c\Z+q)\right),
$$
and
$$
A_2=[-n,n]\setminus\left((a\Z+1)\cup
\bigcup\limits_{c\in\mathscr{C}\setminus\mathscr{C}_0:c\neq a}c\Z\cup \bigcup\limits_{c\in\mathscr{C}_0}(c\Z+q)\right).
$$
Clearly, the sets $A_1$, $A_2$ are $\mathscr{C}$-admissible.

 We are going to prove that $A_1\cup (A_2+m)$ is not $\mathscr{C}$-admissible for any $m\in\Z$.
For sake of contradiction, suppose that $m\in\Z$ is such that $A_1\cup (A_2+m)$ is  $\mathscr{C}$-admissible.
Then there exists $t\in\Z$ such that $a\Z+t$ is disjoint with $A_1\cup (A_2+m)$. Then $(a\Z+t)\cap A_1=\emptyset$, equivalently: $(a\Z+t)\cap[-n,n]\subseteq \bigcup\limits_{c\in\mathscr{C}\setminus\mathscr{C}_0}c\Z \cup \bigcup\limits_{c\in\mathscr{C}_0}(c\Z+q)$, and hence $c|\gcd(a,t)$ for some $c\in\mathscr{C}$ by (\ref{r1}). Since $\mathscr{C}$ is primitive, $c=a$ and $a|t$. Consequently, $a\Z+t=a\Z$.

Since $a\Z\cap (A_2+m)=\emptyset$, it follows that $(a\Z-m)\cap A_2=\emptyset$, equivalently
$$
(a\Z-m)\cap [-n,n]\subseteq (a\Z+1)\cup\bigcup\limits_{c\in\mathscr{C}\setminus\mathscr{C}_0:c\neq a}c\Z \cup \bigcup\limits_{c\in\mathscr{C}_0}(c\Z+q).
$$
If
$$
(a\Z-m)\cap [-n,n]\subseteq \bigcup\limits_{c\in\mathscr{C}\setminus\mathscr{C}_0:c\neq a}c\Z \cup \bigcup\limits_{c\in\mathscr{C}_0}(c\Z+q)
$$
then, by (\ref{r1}), $c|\gcd(a,m)$ for some $c\in\mathscr{C}\setminus\{a\}$. It leads to a contradiction, since $\mathscr{C}$ is primitive.
Thus $(a\Z-m)\cap (a\Z+1)\neq\emptyset$, hence
\begin{equation}\label{r2}
a|m+1.
\end{equation}
Since $\gcd(a,b)>1$, we conclude that
\begin{equation}\label{mb}
b\;\text{does not divide}\;m.
\end{equation}

Since $A_1\cup(A_2+m)$ is $\mathscr{C}$-admissible, there exists $t'\in\Z$ such that $b\Z+t'$ is disjoint with $A_1\cup (A_2+m)$. Then $(b\Z+t')\cap A_1=\emptyset$ and, as above, we prove that $b\Z+t'=b\Z$.
Moreover, $b\Z\cap(A_2+m)=\emptyset$ which is equivalent to
\begin{equation}\label{r3}
(b\Z-m)\cap [-n,n]\subseteq (a\Z+1)\cup\bigcup\limits_{c\in\mathscr{C}\setminus\mathscr{C}_0:c\neq a}c\Z \cup \bigcup\limits_{c\in\mathscr{C}_0}(c\Z+q).
\end{equation}
Furthermore, the numbers $jb-m\notin a\Z+1$ for $j=1,\ldots, L-1$ with $L:=\lcm(a,b)/b$. Indeed, if $jb-m=ax+1$ then, by (\ref{r2}), we obtain  $a|jb$, which is impossible for $j=1,\ldots,L-1$. Moreover, $\lcm(a,b)\Z-m\subseteq a\Z+1$, again by~\eqref{r2}. It follows that
\begin{equation}\label{r4}
(b\Z-m)\setminus (a\Z+1)=\bigcup\limits_{k=1}^{L-1}\big(\lcm(a,b)\Z-m+kb\big).
\end{equation}
Indeed, let $x\in (b\Z-m)\setminus (a\Z+1)$. Then $b|x+m$. There exist $t\in \Z$ and $k\in\{0,1,\ldots,L-1\}$ such that
$$\frac{x+m}{b}=tL+k.$$
Then $x=t\lcm(a,b)+kb-m$ as $L=\frac{lcm(a,b)}{b}$. Observe that $x-1$ is not divisible by $a$ by assumption, hence $k\neq 0$ because of (\ref{r2}). The right-hand side in \eqref{r4} is obviously contained in $b\Z-m$, and if $t\lcm(a,b)-m+kb\in a\Z+1$, then $a|kb$ which we have already noticed being impossible.

By (\ref{r4})  and~(\ref{r3}), for any $k\in\Z$ not divisible by $L$, we have
\begin{equation}\label{r5}
(\lcm(a,b)\Z-m+kb)\cap[-n,n]\subseteq \bigcup\limits_{c\in\mathscr{C}\setminus\mathscr{C}_0:c\neq a}c\Z \cup \bigcup\limits_{c\in\mathscr{C}_0}(c\Z+q).
\end{equation}

Since $L>1$, the subset
$$
A=\{kb-m:k\in\Z, L\;\text{does not divide}\; k\}
$$
of the progression $b\Z-m$ satisfies the assumptions of Lemma \ref{progression}, whence there exists $k\in\Z$ not divisible by $L$
such that
\begin{equation}\label{r6}
|kb-m|\le b.
\end{equation}
By (\ref{r1}) (applied to $r=\lcm(a,b)$, $t=kb-m$ and $\mathscr{C}'=\mathscr{C}\setminus\{a\}$) and~(\ref{r5}), there exists $c\in\mathscr{C}\setminus\{a\}$ such that
$$c|\gcd(\lcm(a,b),kb-m).$$
In particular, $c|kb-m$.
By (\ref{mb}), $kb-m\neq 0$.
 Then $c\le b$ by (\ref{r6}). If $c=b$ then $b|m$ and  we have a contradiction with  (\ref{mb}). Thus $c<b$. Assume that $\gcd(a,c)=1$. Then, as $c|\lcm(a,b)$, we get $c|b$, again a contradiction, since $\mathscr{C}$ is primitive.
Thus $\gcd(a,c)>1$, a contradiction with the choice of $b$, as $a\neq c<b$.
\end{proof}

\subsection{Some consequences}
\paragraph{Proximality and characterization of the Erd\"os case}
It has already been noticed in \cite{Le-Ri-Se}, that whenever $(X_\eta,S)$ is proximal then
\beq\label{entrop1}
h(X_\eta,S)=d\log2,\eeq
where $d$ stands for the density of~1 on $\eta$, i.e., $d=\nu_\eta(C_{\{0\},\emptyset})$.
We recall that in the Erd\"os case and in the Behrend case, the corresponding $\mathscr{B}$-free subshifs are proximal.

\begin{proof} {\em of Corollary~\ref{c:charE}.}
(i) The ``only if'' part is clear. If $X_{\eta}=X_{\mathscr{B}}$, then
 $(X_\eta,S)$ is, by definition,   transitive and the claim follows from Theorem~\ref{t:MS} and \eqref{entrop1}.

(ii) follows from~\eqref{entrop1}.\end{proof}

\paragraph{Tautness and hereditary closure}
We recall that the tautness of $\mathscr{B}$ is equivalent to the fact that the Mirsky measure $\nu_\eta$ has full support \cite{Ke}, \cite{Ku-Le(jr.)}.  Hence, if $\mathscr{B}$ is taut then each block appearing in $\eta$ reappears infinitely often. Recall also that $(\widetilde{X}_\eta,S)$ is transitive if and only if for each block $B$ appearing on $\eta$ there is $B'\geq B$ also appearing in $\eta$ such that $B'$ reappears infinitely many times in $\eta$, see Proposition 3.17 in \cite{Dy-Ka-Ku-Le}. It follows that if $\mathscr{B}$ is taut then its hereditary closure is transitive.

\begin{proof} {\em of Corollary~\ref{c:charEE}.} The proof follows  directly from Theorem~\ref{t:MS}.\end{proof}

Of course all non-coprime thin sets $\mathscr{B}$ satisfy the above assumption so their hereditary closure is a proper subshift of $X_{\mathscr{B}}$, cf.\ \cite{Dy-Ka-Ku-Le} where this is shown for the case of abundant numbers.

\begin{proof} {\em of Corollary~\ref{c:charE1}.} \ By Theorem~\ref{t:MS}, $\mathscr{C}$ is coprime.

(a) If $\mathscr{C}$ is Behrend then the entropy of $(X_{\mathscr{C}},S)$ is zero, and so is $h(X_\eta,S)$. But $(X_\eta,S)$ is proximal (since the all zero sequence is in $X_\mathscr{C}=X_\eta$), so its entropy is the density of~1 in $\eta$, which means that $\mathscr{B}$ is Behrend, so it is not taut, a contradiction.

(b) If $\mathscr{C}$ is finite then the measure of maximal density in $X_{\mathscr{C}}$ is periodic (it is the Mirsky measure given by $\mathscr{C}=\{c_1,\ldots,c_m\}$). It must be then equal to $\nu_\eta$.
This periodic measure in $X_\mathscr{C}$ has no full support, so $\nu_\eta$ has no full support in $X_\eta$ which contradicts the tautness of $\mathscr{B}$.

(c) We now have that $\mathscr{C}$ is Erd\"os. It follows that $X_{\mathscr{C}}=X_{\eta_{\mathscr{C}}}$, where $\eta_{\mathscr{C}}=\raz_{\cf_{\mathscr{C}}}$, and, by our assumption,
$$
X_{\eta_{\mathscr{C}}}=X_\eta.$$
Since now $\mathscr{C}$ and $\mathscr{B}$ are both taut, the assertion follows from \cite{Dy-Ka-Ku-Le}.\end{proof}

\paragraph{Minimality and three subshifts}
\begin{proof} {\em of Corollary~\ref{c:charEEE}} Clearly, $\widetilde{X}_\eta$ is not minimal (it has a fixed point). Moreover,
$(\widetilde{X}_\eta,S)$ is still transitive, since $(X_\eta,S)$ is a Toeplitz system by \cite{Dy-Ka-Ku-Le} (in fact, by \cite{Ke}, $\eta$ itself is a Toeplitz sequence), while $(X_{\mathscr{B}},S)$ is not transitive.
\end{proof}

\subsection{Isolated points}\label{s:isolated}
\begin{proposition}\label{p:isol}
For any $\mathscr{B}$, the $\mathscr{B}$-admissible subshift $(X_{\mathscr{B}},S)$ has no isolated points.\end{proposition}
\begin{proof} First note that if $y\in X_{\mathscr{B}}$ and ${\rm supp}(y)$ is infinite then we can approximate $y$ by $y^{(n)}\in X_{\mathscr{B}}$, where $y^{(n)}$ arises from $y$ by changing one 1 into 0 at a position $k_n$, where $k_n\to\infty$. It follows that such a $y$ is not isolated.

So the problem is to show that finite support sequences are not isolated points in $X_{\mathscr{B}}$.
To see this we need to show that if $A\subset\Z$ is finite and $\mathscr{B}$-admissible, then there exists $m$ arbitrarily large such that $A\cup\{m\}$ is still $\mathscr{B}$-admissible. Let $\mathscr{B}'=\{b_1,\ldots,b_n\}=\{b\in\mathscr{B}:b\le |A|+1\}$. It is not hard to see that if we take $a\in A$ and consider $m=a+xb_1\ldots b_n$ then $A\cup\{m\}$ is still $\mathscr{B}'$-admissible and hence it is $\mathscr{B}$-admissible.

\end{proof}

In contrast to $(X_{\mathscr{B}},S)$, the $\mathscr{B}$-free subshift  $(X_\eta,S)$ has often isolated points. Recall that if we consider $\mathscr{B}=\{pq:\:p,q\in\PP\}$ (so called {\em the set of semi-primes}), then $X_\eta$ is the subshift of prime numbers. Hence
$$
  X_{\raz_{\PP\cup\{1\}}}\subset X_{\{pq:\:p,q\in \PP\}}.$$
The subshift on the right-hand side is uncountable (all admissible subshifts so are \cite{Dy-Ka-Ku-Le}) and it is an open problem whether the subshift on the left-hand side  is uncountable. We might ask whether these two subshifts are equal but this is not the case, as the subshift on the right has no isolated points by Proposition~\ref{p:isol}, while the subshift on the left clearly has such points. Indeed, in $\eta=\raz_{\PP\cup\{1\}}$ there are blocks that reappear only finitely many times  (e.g.\ the block\footnote{We denote a 0-1-sequence $(x_1,\ldots,x_n)$ by  $x_1\ldots x_n$.} \ 11 has this property), see \cite{Dy-Ka-Ku-Le}.

\begin{remark}\label{r:isoladm} For each (primitive) set $\mathscr{B}$, we have
\beq\label{isoladm1}
(X_{\raz_{\mathscr{B}}})_0\subset X_{\mathscr{B}}.\eeq
Indeed, if a block $C$ appears in  $y\in (X_{\raz_{\mathscr{B}}})_0$, then it appears infinitely often on  $\raz_{\mathscr{B}}$.   Then the sets ${\rm supp}(C)+k$ appear as the supports of the block $C$ for infinitely many $k$.  However, if $|k|$ is large enough then ${\rm supp}(C)+k$ is $\mathscr{B}$-admissible, since $\mathscr{B}$ is primitive (if ${\rm supp}(C)=\{b_{i_1},\ldots,b_{i_r}\}\subset \mathscr{B}$ is not $\mathscr{B}$-admissible then for some $b\in\mathscr{B}$, ${\rm supp}(C)$ mod~$b$ gives all residue classes mod~$b$. If for some $k\geq1$, also ${\rm supp}(C)+k\subset \mathscr{B}$, then for some $1\leq j\leq r$ we must have $b_{i_j}+k=0$ mod $b$, that is, $b|b_{i_j}+k\in\mathscr{B}$; by primitivity, $b_{i_j}+k=b$, so $k$ stays bounded).\end{remark}

\subsection{Minimal complexity of $\mathscr{B}$-admissible subshifts}\label{s:mincompl}
As we have already mentioned, $\mathscr{B}$-admissible subshifts are uncountable (this also follows from Proposition~\ref{p:isol}). Another observation (which is a consequence of Proposition~\ref{p:staszek}) is that the complexity of such subshifts is always superpolynomial, in fact, that Proposition~\ref{p:staszek1} holds:

\begin{proof} {\em of Proposition~\ref{p:staszek1}.} \
As $X_{\PP}\subseteq X_{\mathscr{B}}$ (Proposition~\ref{p:staszek}), it is enough to prove the assertion for $\mathscr{B}=\PP$. Let us first observe  that, for any $N$, the set
$$
(N,2N]\cap \PP
$$
is $\PP$-admissible (and  so every subset of it is $\PP$-admissible). Indeed, denote the intersection
$(N,2N]\cap\PP$ as $p_1<\ldots<p_k$.  If now $p>N$ then the cardinality of this set is strictly smaller than $p$. So take $p\leq N$. Then the zero residue class modulo~$p$ is not encountered in $\{p_1,\ldots, p_k\}$. It follows that the set $\{p_1,\ldots,p_k\}$ is $\PP$-admissible.

Moreover (see (\ref{sza5}))
\begin{equation}
|[N+1,2N]\cap \PP|=\pi(2N)-\pi(N)=\frac N{\log N}(1+o(1))
\end{equation}
It follows that the lower bound of the number of $\PP$-admissible blocks of length $N$ is $2^{\frac N{\log N}(1+o(1))}=(2+o(1))^{\frac{N}{\log(N)}}$.
\end{proof}

\begin{remark}\label{r:og1} Note that the argument used in the proof of Proposition~\ref{p:staszek1} can be applied in the general context. Recall that $\mathscr{B}$ is always assumed to be primitive. Let
$$
e_N:=|(N,2N]\cap\mathscr{B}|,\;N\geq1.$$
Note that the set $(N,2N]\cap\mathscr{B}$ is always $\mathscr{B}$-admissible since, whenever $b\leq N$,  $b$ cannot divide any number in $(N,2N]\cap\mathscr{B}$ by primitivity. Hence, we obtain
\beq\label{og1}
{\rm cpx}_{X_\mathscr{B}}(n)\geq 2^{e_n}.\eeq
This argument applied to the set $\PP_k$ of $k$-almost prime numbers yields
$$
{\rm cpx}_{X_{\PP_k}}(n)\geq 2^{\frac{n(\log\log n)^{k-1}}{(k-1)!\log n}(1+o(1))}, \;\;n\geq 1.
$$
However, this lower bound of the complexity via an asymptotic density of the set $\mathscr{B}$ itself is rather weak for sets $\mathscr{B}$ which are very sparsed. If $\mathscr{B}=\{p^2:\:p\in\PP\}$ then $e_n/n\to0$, while the entropy of $X_\mathscr{B}$ is positive (as $\mathscr{B}$ is Erd\"os).
\end{remark}

Remembering that we consider only sets $\mathscr{B}$ which are primitive, the method which we used above yields the following classical result:

\begin{corollary}\label{c:BerZero}
Each Behrend set $\mathscr{B}$ has upper Banach density zero.
\end{corollary}
\begin{proof} Let us call an interval $[M,M+N]$ ``good'' if $N\leq M$. Suppose that $BD^\ast(\mathscr{B})=c>0$. Let us notice that for $\vep>0$ small enough if for an interval $[M,M+N]$ we have $|[M,M+N]\cap\mathscr{B}|\geq (1-\vep)cN$ then by dividing $[M,M+N]$ into small consecutive intervals of length $x$ (with $x\to\infty$ when $N\to\infty$), for at least one interval $[M+\ell_0 x,M+(\ell_0+1)x]$ with $\ell_0\geq2$, we obtain
$$
[M+\ell_0 x,M+(\ell_0+1)x]\text{ is ``good'' and } |[M+\ell_0 x,M+(\ell_0+1)x]\cap\mathscr{B}|\geq (1-\vep)\frac c2 x.$$
Since now $[M+\ell_0 x,M+(\ell_0+1)x]\cap \mathscr{B}$ is $\mathscr{B}$-admissible (since $\mathscr{B}$ is primitive), the subshift $X_{\mathscr{B}}$ contains at least $2^{(1-\vep)\frac c2 x}$ $\mathscr{B}$-admissible blocks of length $x$, clearly the entropy of the subshift $(X_{\mathscr{B}},S)$ is at least $(1-\vep)\frac c2>0$, which is in contradiction with the fact that Behrend $\mathscr{B}$-admissible subshifts are of entropy zero \cite{Dy-Ka-Ku-Le}.\footnote{Another ``dynamical'' proof can be obtained on the basis of two observations:
(i)  $\mathscr{B}\subset \cf_{\{bc:\: b,c\in\mathscr{B},b\neq c\}}$ for each set $\mathscr{B}$;\\
(ii) The set $\{bc:\: b,c\in\mathscr{B},b\neq c\}$ is Behrend if $\mathscr{B}$ was, see Proposition~\ref{p:marsta}.

Then, we finish the proof in the same way, noticing that $\cf_{\{bc:\: b,c\in\mathscr{B},b\neq c\}}$ has upper Banach density~0 in view of \cite{Dy-Ka-Ku-Le}.}

\end{proof}

\begin{remark} It also follows that in Remark~\ref{r:og1} we should consider $e_n$ as the numbers corresponding to the upper Banach density.\end{remark}

\section{Behrend set which yields a $\mathscr{B}$-free subshift which is $\mathscr{B}$-admissible. Proof of Theorem~\ref{t:BehrendA}}

\begin{lemma}\label{admissible} Let $m\in\N$.
Assume that $\mathscr{C}$ is a finite coprime set containing at least $m$ elements greater than $m$. Then every $\mathscr{C}$-admissible block of length $m$ appears on $\eta_{\mathscr{C}}=\1_{\cF_{\mathscr{C}}}$.
\end{lemma}

\begin{proof}
Let $\mathscr{C}'=\{c\in \mathscr{C}:c\le m\},\; \mathscr{C}''=\{c\in \mathscr{C}:c> m\}$.
Let $B=x_1x_2\ldots x_m$ be an admissible block. By the admissibility of $B$, for every $c\in \mathscr{C}'$ there exists $r_c\in\{1,\ldots,m\}$ such that $x_j=0$ for every  $j=1,\ldots m$ satisfying $j\equiv r_c \mod c$.

We define $r_c\in\{0,1,\ldots,m\}$ for $c\in\mathscr{C}''$ in the following way. We set $r_c=0$ for every $c\in\mathscr{C}''$ in the case $B=11\ldots 1$. Otherwise, we can choose $r_c\in\{1,\ldots,m\}$ in such a way that
$$
\{j\in\{1,\ldots,m\}:x_j=0\}=\{r_c:c\in \mathscr{C}''\}.
$$
It is possible since $|\mathscr{C}''|\ge m$.

Note that $x_{r_c}=0$ whenever $r_c\ge 1$, for every $c\in\mathscr{C}$.

By the Chinese Remainder Theorem there exists $n\in\N$ such that
\begin{equation}\label{congruence}
n\equiv-r_c\mod c
\end{equation}
for every $c\in \mathscr{C}$. We show that $\eta_{\mathscr{C}}[n+1,n+m]=B$.

If $x_j=0$ then $j=r_c$ for some $c\in\mathscr{C}''\subseteq \mathscr{C}$. Then $c|n+j$ by (\ref{congruence}) and $\eta_{\mathscr{C}}[n+j]=0$.

    Assume  that $\eta_{\mathscr{C}}[n+j]=0$, that is, $c|n+j$ for some $c\in\mathscr{C}$. By (\ref{congruence}) it is equivalent to  $j\equiv r_c\mod c$. If $c\in \mathscr{C}'$ then $x_j=0$ by the choice of $r_c$. If $c\in\mathscr{C}''$ then $c> m$ and since $1\le j\le m$, $0\le r_c\le m$ it follows that $r_c=j\ge 1$, thus $x_j=0$.
\end{proof}

Now we turn to the proof of Theorem~\ref{t:BehrendA}.

\begin{proof} {\em (of Theorem~\ref{t:BehrendA})} We fix $0<\varepsilon <1$. We construct: sequences $2=N_0<N_1<\ldots <N_l<\ldots$ and $M_1<M_2<\ldots <M_l<\ldots$ of numbers and a sequence
$\{2\}=\mathscr{B}_0\subset \mathscr{B}_1\subset\ldots \subset \mathscr{B}_l\subset\ldots $ of finite sets of prime numbers such that for every $l\in\N\cup\{0\}$:
\begin{enumerate}
\item[(i)$_l$] $N_l\ge l+1$,
\item[(ii)$_l$] $\mathscr{B}_l\subseteq [1,N_l]\cap\PP$,
\item[(iii)$_l$] $\mathscr{B}_l$ contains at least $l+1$ elements greater than $l+1$,
\item[(iv)$_l$] $\prod\limits_{p\in\mathscr{B}_l}\left(1-\frac{1}{p}\right)\le \varepsilon^l$,
\item[(v)$_l$] $\mathscr{B}_{l}\cap [1,N_{l-1}]=\mathscr{B}_{l-1}$ if $l>0$,
\item[(vi)$_l$] every $\mathscr{B}_l$-admissible block of length $l$ appears in $\eta_l[1,N_l]$, where $\eta_l=\1_{\cF_{\mathscr{B}_l}}$,
\item[(vii)$_l$] $N_{l-1}<M_{l}<N_{l}$ if $l>0$,
\item[(viii)$_l$] $\lcm(\mathscr{B}_l)\mid M_{l+1}-N_l$ and $\frac{\rho(M_{l+1}-N_l)}{M_{l+1}-N_l}\le d(\cF_{\mathscr{B}_l})$ if $l>0$.
\end{enumerate}
Observe that with $N_0=2$ and $\mathscr{B}_0=\{2\}$ the conditions (i)$_0$-(viii)$_0$ are satisfied.
Assume that $l\ge 0$ and $2=N_0<N_1<\ldots <N_l$, $M_1<\ldots<M_{l}$ and $\{2\}=\mathscr{B}_0\subset \mathscr{B}_1\subset\ldots \subset \mathscr{B}_l$ have been defined and the conditions (i)$_l$-(viii)$_l$ are satisfied. We will define $M_{l+1}$, $\mathscr{B}_{l+1}$ and $N_{l+1}$.

Since, by (iii)$_l$, $\mathscr{B}_l$ is a finite set of prime numbers containing at least $l+1$ elements greater than $l+1$, it follows by Lemma~\ref{admissible} that every $\mathscr{B}_l$-admissible block of length $l+1$ appears in $\eta_{l}$.  Let $M_{l+1}>N_l$ be a number large enough that every $\mathscr{B}_l$-admissible block of length $l+1$ appears in $\eta_{l}[1,M_{l+1}]$.  Since $\frac{\rho(n)}{n}\searrow 0$ as $n\rightarrow +\infty$ and $d(\cF_{\mathscr{B}})>0$, we can chose $M_{l+1}$ such that the condition (viii)$_{l+1}$ is satisfied.

There exist prime numbers $p_1<p_2<\ldots<p_T$, greater than $M_{l+1}$ and such that
\begin{equation}\label{nierownosc}
\prod_{i=1}^T\left(1-\frac{1}{p_i}\right)\le \varepsilon.
\end{equation}

We can also assume that $T\ge l+2$.

We set $N_{l+1}=p_T$ and $\mathscr{B}_{l+1}=\mathscr{B}_l\cup\{p_1,p_2,\ldots,p_T\}$. Then the condition (vii)$_{l+1}$ is satisfied. Observe that
\begin{equation}\label{aux}
[1,M_{l+1}]\cap \mathscr{B}_{l+1}=\mathscr{B}_l.
\end{equation}

We show that the conditions (i)$_{l+1}$ - (viii)$_{l+1}$ hold. Indeed:
\begin{enumerate}
\item[(i)$_{l+1}$] Clear by the construction since $N_{l+1}>M_l\ge N_l$ and $N_l\ge l+1$ by the induction hypothesis (i)$_{l}$.
\item[(ii)$_{l+1}$]  Follows by the choice of $\mathscr{B}_{l+1}$ and $N_{l+1}$.
\item[(iii)$_{l+1}$] Follows since $T\ge l+2$ and $p_1>M_{l+1}\ge N_l\ge l+1$.
\item[(iv)$_{l+1}$] We write
$$
\prod\limits_{p\in\mathscr{B}_{l+1}}\left(1-\frac{1}{p}\right)=
\prod\limits_{p\in\mathscr{B}_l}\left(1-\frac{1}{p}\right)\cdot \prod_{i=1}^T\left(1-\frac{1}{p_i}\right)
$$
The assertion follows from the induction hypothesis (iv)$_l$ and (\ref{nierownosc}).
\item[(v)$_{l+1}$] Follows by (\ref{aux}) and $(ii)_{l+1}$  since $M_{l+1}> N_l$.
\item[(vi)$_{l+1}$] Let $B$ be a $\mathscr{B}_{l+1}$-admissible block of length $l+1$. Then $B$ is $\mathscr{B}_l$-admissible, hence $B$ appears in $\eta_l[1,M_{l+1}]$ by the choice of $M_{l+1}$. Moreover, $\eta_l[1,M_{l+1}] =\eta_{l+1}[1,M_{l+1}]$ by (\ref{aux}). Consequently, $B$ appears in $\eta_{l+1}[1,N_{l+1}]$ as $N_{l+1}>M_{l+1}$.
\end{enumerate}
The conditions (vii)$_{l+1}$ and (viii)$_{l+1}$ are discussed before.

We set $\mathscr{B}=\bigcup_{l=1}^{\infty}\mathscr{B}_l$. Observe that $[1,N_l]\cap \mathscr{B}= \mathscr{B}_l$  by (v)$_{l+1}$ and thus
\begin{equation}\label{fin}
\eta_l[1,N_l]=\eta[1,N_l]
\end{equation}
for every $l\in \N$.

Assume that $B$ is a $\mathscr{B}$-admissible block of length $l$. By the condition (vi)$_l$, $B$ appears in $\eta_l[1,N_l]$, hence it appears in $\eta$ by (\ref{fin}).

By the construction, $\mathscr{B}\subseteq\PP$. Moreover, $\mathscr{B}$ is Behrend because
$$
\prod_{p\in\mathscr{B}}\left(1-\frac1p\right)=0
$$
thanks to (iv)$_l$, $l\ge 0$.

It remains to prove the statement on the complexity. By the construction of the set $\mathscr{B}$  (see (\ref{aux}) we see that
\begin{equation}\label{cpx21}
\cF_{\mathscr{B}}\cap[N_l+1,M_{l+1}]=\cF_{\mathscr{B}_l}\cap[N_l+1,M_{l+1}].
\end{equation}
As $\lcm(\mathscr{B}_l)|M_{l+1}-N_l$ by (ix)$_{l}$, we have
\begin{equation}\label{cpx31}
|\cF_{\mathscr{B}_l}\cap[N_l+1,M_{l+1}]|=(M_{l+1}-N_l)d(\cF_{\mathscr{B}_l}).
\end{equation}
By (\ref{cpx21}), (\ref{cpx31}) and (ix)$_l$ we get
\begin{equation}
|\cF_{\mathscr{B}}\cap[N_l+1,M_{l+1}]|\ge \rho(M_{l+1}-N_l).
\end{equation}
Every subset of the set $\cF_{\mathscr{B}}\cap[N_l+1,M_{l+1}]$ is $\mathscr{B}$-admissible, thus there are at least $2^{\rho(M_{l+1}-N_l)}$ $\mathscr{B}$-admissible blocks and each of them appears in $\eta$.   It follows that ${\rm cpx}_{\eta}(M_l-N_l)\ge 2^{\rho(M_l-N_l)}$.
\end{proof}

\section{Complexity of the subshifts of primes and semi-primes}

\subsection{Complexity of the subshift of primes - proof of Theorem~\ref{t:complexityP}}\label{s:TTao}
Let us start with the upper bounds. Observe that by the arguments in Remark \ref{r:og1},
\begin{equation}\label{eq:sec51}
\text{the block}\; \raz_{\PP}(n,n+N]\;\text{ is}\; \PP\text{-admissible for}\; n\ge N.
\end{equation}
It follows that (\ref{tao}) is a consequence of the upper bound in (\ref{tao1}).

Let $N$ be a fixed positive integer and consider $n\geq \sqrt{N}+1$. Set
\beq\label{sza1}
B_n=B_{n,N}:=(\raz_{\PP}(n+1),\ldots,\raz_{\PP}(n+N))\in\{0,1\}^N.\eeq
We denote by
\beq\label{sza2}
O_{n,N}:=\{1\leq i\leq N:\: \raz_{\PP}(n+i)=1\}\eeq
the support of $B_n$. Note also that
\beq\label{sza3}
O_{n,N}\subset\{1\leq i\leq N:\: p\not| n+i\;(\forall p\leq\sqrt{N})\}=\{1\leq i\leq N:\:i\not\equiv-n\;(\text{mod }p)\;\forall p\leq\sqrt{N}\}=:S_{n,N}.\eeq
The set $S_{n,N}$ depends solely on the residue classes of $n$ mod~$p$, $p\leq\sqrt{N}$; indeed, if we set $\PP\cap[1,\sqrt{N}]=\{p_1,\ldots, p_k\}$ then for any $n$, there are $0\leq m_j<p_j$ such that
$$
S_{n,N}=\{1\leq i\leq N:\:i\not\equiv m_j\text{ (mod }p_j)\; \forall 1\leq j\leq k\}=:C_N(m_1,\ldots,m_k)=C(-n\text{ (mod }p_1),\ldots, -n\text{ (mod }p_k)).$$
Therefore, each set $S_{n,N}$, $n\geq\sqrt{N}+1$, must be one of the sets $C_N(m_1,\ldots,m_k)$ and the number of the latter sets is
$$
\prod_{j=1}^k p_j=\prod_{p\leq\sqrt{N}}p.$$
 This and \eqref{sza3} imply that the support of each $B_{n,N}$, $n\geq \sqrt{N}+1$, is a subset of one of the sets $C_N(m_1,\ldots,m_k)$.

Suppose that there exists $K=K_N$ such that for any $0\leq m_j<p_j$, $j=1,\ldots,k$, we have $|C_N(m_1,\ldots,m_k)|\leq K$. Hence, each  of the sets $C_N(m_1,\ldots,m_k)$ has at most $2^K$ subsets. It follows that the supports of the blocks $B_{n,N}$, $n\geq\sqrt{N}+1$, can give at most $\left(\prod_{p\leq\sqrt{N}}p\right)\cdot 2^K$ subsets. Therefore,

\begin{equation}\label{eq:an}
\begin{array}{l}
A_N:=|\{B_{n,N}:\:n\in\N\}|=|\{\mbox{pattern vectors of length $N$ of primes} \}|\leq\\
|\{B_{n,N}:\: n\leq\sqrt{N}\}|+|\{B_{n,N}:\:n>\sqrt{N}\}|\leq \sqrt{N}+\left(\prod_{p\leq\sqrt{N}}p\right)\cdot 2^K.
\end{array}
\end{equation}
Now, by the Prime Number Theorem,
\begin{equation}\label{eq:prodp}
\prod_{p\leq\sqrt{N}}p=\exp(\sum_{p\leq\sqrt{N}}\log p)=\exp(\sqrt{N}(1+o(1))).
 \end{equation}
 On the other hand, $K$ is just an upper bound for the maximal number of elements in the set $\{1,2,\ldots,N\}$ after deleting from it a particular residue class $0\leq m_j<p_j$ for all $j=1,\ldots, k$  (that is, for all $p\leq \sqrt{N}$).

In the language of the Large Sieve (see \eqref{LS1}), we define $C=C_N(m_1,\ldots,m_k)\subset[1,N]$, $\omega(p)=1$ for $p\leq\sqrt{N}$ and $\omega(p)=0$ for all $p>\sqrt{N}$. Therefore, whenever $R<\sqrt{N}$, we have $\sigma(R)=\sum_{q\leq R}\frac{\mob^2(q)}{\phi(q)}$  (remembering that $\mob^2$  is the indicator of the set of  square-free numbers, and $\phi$ stands for the Euler function). However, it is classical that $\sum_{q\leq R}\frac{\mob^2(q)}{\phi(q)}>\log R$ (in fact, it is $=\log R+O(1)$; see \eqref{NT:pnt5}). If we now set $R=\sqrt{N}/\log N$, we obtain
\begin{equation} \label{osz}
|C_N(m_1,\ldots,m_k)|\leq \frac{2N}{\log N}(1+o(1))=:K_N.
\end{equation}
Finally, by (\ref{eq:an}) and (\ref{eq:prodp}), we obtain
$$
A_N\leq \sqrt{N}+\exp(\sqrt{N}(1+o(1)))2^{\frac{2N}{\log N}(1+o(1))}\ll$$$$
2^{\frac{\sqrt{N}}{\log 2}(1+o(1))+\frac{2N}{\log N}(1+o(1))}=
2^{\frac{2N}{\log N}(1+o(1))}=(4+o(1))^{N/\log N}$$
and (\ref{tao1}) follows.


Let us go now to the lower bounds. By the  Prime Number Theorem the number of primes in $[N+1,2N]$ is
\beq\label{sza5}
\pi(2N)-\pi(N)=\left(\frac{2N}{\log(2N)}-\frac N{\log N}\right)(1+o(1))=\frac N{\log N}(1+o(1)).\eeq
Consider $k$ primes in $[N+1,2N]$: $p_1<\ldots<p_k$. Then $k\leq N$. By (\ref{eq:sec51}), the set $\{p_1,\ldots,p_k\}$ is  $\PP$-admissible.

Thus, there exists a $\PP$-admissible block of length $N$ with support of cardinality at least $\frac N{\log N}(1+o(1))$, hence  (\ref{tao1}) follows.

 Moreover, assuming the Hardy-Littlewood conjecture\footnote{See Theorem X in \cite{Ha-Li}.}, for any $X$ large enough, we have
\begin{equation}\label{H-L}
\sum_{n\leq X}\raz_{\PP}(n+p_1)\cdot\ldots\cdot \raz_{\PP}(n+p_k)\sim \mathfrak{C}\frac X{\log^k(X)}.
\end{equation}
with the actual constant $\mathfrak{C}=\mathfrak{C}(p_1,\ldots,p_k)$ positive, since we consider the admissible case.

Moreover, by \cite[Cor. 3.14 \& Eq. (3.46)]{Mont-Vaug2}, it is known that
\begin{equation}\label{HL-upper}
\sum_{n\leq X}\raz_{\PP}(n+p_1)\cdot\ldots\cdot \raz_{\PP}(n+p_k)\raz_{\PP}(n+i)\leq 2^{k+1} (k+1)! \mathfrak{C}_i\frac X{\log^{k+1}(X)}(1+o(1)).
\end{equation}
where  $\mathfrak{C}_i=\mathfrak{C}(p_1,\ldots,p_k,i)$, regardless of whether or not $\{p_1,\ldots,p_k,i\}$ is admissible (if it is not, then $\mathfrak{C}_i=0$ and the bound \eqref{HL-upper} is trivial).
Therefore,
$$
\sum_{n\leq X;\raz_{\PP}(n+i)=0\;(\forall i\in[N+1,2N]\setminus\{p_1,\ldots,p_k\})}\raz_{\PP}(n+p_1)\cdot\ldots\cdot \raz_{\PP}(n+p_k)=$$$$
\sum_{n\leq X}\raz_{\PP}(n+p_1)\cdot\ldots\cdot \raz_{\PP}(n+p_k)
- \sum_{n\leq X;\raz_{\PP}(n+i)=1\;\text{ for some } i\in[N+1,2N]\setminus\{p_1,\ldots,p_k\})}\raz_{\PP}(n+p_1)\cdot\ldots\cdot \raz_{\PP}(n+p_k)\gg_k$$
$$\mathfrak{C}\frac{X}{\log^k(X)}-N2^{k+1}k!\max_{i\in[N+1,2N]\setminus\{p_1,\ldots,p_k\}}\{\mathfrak{C}_i\}\frac{X}{\log^{k+1}(X)}\gg_{k,N}\frac{X}{\log^k(X)}-\frac{X}{\log^{k+1}(X)}.$$
Thus, for $X$ large enough the right-hand side above is positive and there exists $n^\ast$ such that the numbers  $n^\ast+p_1,\ldots,n^\ast+p_k$ are all primes and all the $n^\ast+i$ for $i\in[N+1,2N]\setminus\{p_1,\ldots,p_k\}$ are not.

We have just shown that there is an injection from the set $T_N$ of the tuples $\{p_1,\ldots, p_k\}\subset \{N+1,\ldots,2N\}$ to the set ${A}_N$ of blocks of length $N$ of the indicator of the primes.  Therefore, by \eqref{sza5}, $2^{|T_N|}=2^{\frac N{\log N}(1+o(1))}=(2+o(1))^{\frac N{\log N}}\ll |A_N|$ and (\ref{tao2}) follows.


\subsection{Consequences of Dickson's conjecture: complexity for the subshift of semi-primes (a lower bound in Theorem~\ref{t:SP})} \label{s:DC2}

In this section we collect some consequences of Dickson's conjecture, first mentioned in  \cite{Dickson}.
Below we formulate the conjecture following \cite{Zhang}.

Consider a finite family $\Phi=\{f_i(x)=a_i+b_ix:i=1,\ldots,k\}$ of  linear polynomials with  integer coefficients $a_i$ and $b_i$, $b_i\ge 1$ (for $i=1,\ldots,k$). We say that $\Phi$ satisfies {\em Dickson's condition} if
$$
\forall_{p\in\PP}\exists_{y\in\Z}\;\prod_{i=1}^kf_i(y)\nequiv 0\mod p.
$$

\begin{remark} \label{r:dopDick} Note that if $A\subset\Z$ is finite, then $A$ is $\PP$-admissible if and only if the family $\{x-a:\:a\in A\}$ satisfies Dickson's condition.\end{remark}

\begin{Conj}  \label{Dickson} (Dickson's conjecture)
If $\Phi$ satisfies Dickson's condition, then there exist infinitely many natural  numbers $m$ such that all the numbers $f_1(m),\ldots,f_k(m)$ are primes.
\end{Conj}

For future use let us prove the following technical lemma.

\begin{Lemma}\label{lem:condition}
Let $\Phi=\{f_i(x)=a_i+b_ix:i=1,\ldots,k\}$, where  $a_i, b_i\in\Z$, $b_i\ge 1$, for $i=1,\ldots,k$. Assume that
\begin{enumerate}
\item[(a)] $\gcd(a_i,b_i)=1$ for $i=1,\ldots,k$ and
\item[(b)] $|\{i\in\{1,\ldots,k\}: p\nmid b_i\}|<p$ for every $p\in \PP$.
\end{enumerate}
Then $\Phi$ satisfies Dickson's condition.
\end{Lemma}

\begin{proof}
Fix $p\in \PP$ and let $\{i_1,\ldots,i_l\}=\{i\in\{1,\ldots,k\}: p\nmid b_i\}$.
If $l=0$, that is, $p|b_i$ for $i=1,\ldots,k$, then in view of (a), $p\nmid a_i$ for $i=1,\ldots,k$, so $p$ does not divide $\prod_{i=1}^k(a_i+b_iz)$ for any $z\in\Z$. From now on we assume that $l>0$. Denote $b=\lcm(b_{i_1},\ldots,b_{i_l})$. Since, by the assumption (b), $l<p$, there exists $a'\in\Z$
such that
\begin{equation}\label{e71}
a'\nequiv a_{i_j}\frac{b}{b_{i_j}}\mod p\;\text{for}\; j=1,\ldots,l.
\end{equation}
As $\gcd(b,p)=1$, there exists $y\in\Z$ such that
\begin{equation}\label{eq:bya}
by\equiv-a'\mod p.
\end{equation}
Suppose that $p|a_i+b_iy$ for some $i\in\{1,\ldots,k\}$. If $p|b_i$ then, by the assumption (a), $p\nmid a_i$, a contradiction. Thus $i=i_j$ for some $j\in\{1,\ldots,l\}$.  We have
$$
b_{i_j}y\equiv-a_{i_j}\mod p.
$$
Multiplying by $\frac{b}{b_{i_j}}$ we obtain
$$
by\equiv-a_{i_j}\frac{b}{b_{i_j}}\mod p.
$$
In view of (\ref{eq:bya}), $p|a'-a_{i_j}\frac{b}{b_{i_j}}$, a contradiction with~\eqref{e71}. We have shown that $p$ does not divide $\prod_{i=1}^k(a_i+b_iy)$.
\end{proof}

It is known that Dickson's conjecture \ref{Dickson} is equivalent to the following stronger statement.

\begin{Conj} \label{Dickson*} (Dickson's conjecture*)
Consider finite families $\Phi=\{f_i(x)=a_i+b_ix:i=1,\ldots,k\}$ and $\Gamma=\{g_j(x)=c_j+d_jx:j=1,\ldots,k'\}$ with $a_i,b_i,c_j,d_j\in\Z$, $b_i,d_j\ge 1$, for  $i=1,\ldots,k, j=1,\ldots, k'$. Assume moreover that $\Phi\cap(-\Gamma\cup\Gamma)=\emptyset$, that is, $\pm g_j(x)\notin\Phi$ for every $j=1,\ldots,k'$. If $\Phi$ satisfies Dickson's condition, then there exist infinitely many natural  numbers $m$ such that all the numbers $f_1(m),\ldots,f_k(m)$ are primes and all the numbers $g_1(m),\ldots,g_{k'}(m)$ are composite.
\end{Conj}

The nontrivial implication Conjecture \ref{Dickson}$\Rightarrow$ Conjecture \ref{Dickson*} follows from
the arguments of A. Schinzel from \cite{Sch61} (the proof of $H\Rightarrow C_{13}$). For  convenience of the reader we provide an alternative proof.

\begin{Lemma}\label{lem:div}
Suppose that $a_i,b_i,c,d\in\Z$, $b_i,d\ge 1$, where $i=1,\ldots,k$ and $\gcd(c,d)=1$. Assume that the polynomial $a_i+b_ix$ is not divisible\footnote{in the ring $\Z[x]$.} by $c+dx$ for any $i=1,\ldots,k$.
Then there exists infinitely many $z\in\N$ such that $c+dz$ is a prime and it does not divide
$\prod_{i=1}^k(a_i+b_iz)$.
\end{Lemma}

\begin{proof}
By the Dirichlet Theorem, $c+dz$ is a prime for infinitely many integers $z$. Assume for a contradiction that for almost every such $z$:
$$
c+dz|\prod_{i=1}^k(a_i+b_iz).
$$
It follows that there exist $i\in\{1,\ldots,k\}$ and an infinite increasing sequence $(z_j)_j$ of integers such that $c+dz_j$ is prime and it divides $a_i+b_iz_j$ for every $j$.
Since
$$
\frac{a_i+b_iz_j}{c+dz_j}\rightarrow \frac{b_i}{d}\;\text{as}\; j\rightarrow\infty,
$$
the sequence $( \frac{a_i+b_iz_j}{c+dz_j})_j$ (of integers) is eventually constant, whence $c+dx|a_i+b_ix$, a contradiction. The assertion follows.
\end{proof}

\begin{Lemma}\label{lem:guard}
Assume that a family $\Phi=\{f_i(x)=a_i+b_ix:\:i=1,\ldots,k\}$ satisfies Dickson's condition and $g(x)=c+dx$ for some $c,d\in\Z$, $d\ge 1$. Assume that $\pm g(x)\notin \Phi$. Then, for any $L\in\N$, there exists a finite family $\Psi=\{-c'_j+x:j=1,\ldots,t\}$, where  $t\in\N\cup\{0\}$ and  $c'_1,\ldots,c'_t\in\Z$,  $c'_1,\ldots,c'_t\ge L$, such that the family $\Phi\cup\Psi$ satisfies Dickson's condition, whereas  the family $\Phi\cup\Psi\cup\{g(x)\}$ does not.
\end{Lemma}

\begin{proof}
If $\Phi\cup\{g(x)\}$ does not satisfy Dickson's condition, then we set $\Psi=\emptyset$. This is always the case if $c$ and $d$ are not coprime. From now on we assume that $\gcd(c,d)=1$ and $\Phi\cup\{g(x)\}$  satisfies Dickson's condition.

Clearly, $\gcd(a_i,b_i)=1$ for $i=1,\ldots,k$, hence (as $\pm g(x)\notin\Phi$) the polynomial $g(x)$ does not divide any of $f_i(x)$.
By Lemma \ref{lem:div}, there exists $z\in\N$ such that $p=c+dz$ is a prime which is greater than any of $b_i$, $i=1,\ldots,k$, greater than $d$,  and such that
\begin{equation}\label{eq:cond0}
\prod_{i=1}^k(a_i+b_iz)\nequiv 0\mod p.
\end{equation}


Since $\Phi\cup\{g(x)\}$  satisfies Dickson's condition, the set
$$
R=\{r\in\Z: a_i+b_ir\notin p\Z\;\text{for}\;i=1,\ldots,k\;\text{and}\; c+dr\notin p\Z\}
$$
is nonempty. As $p=c+dz$,
\begin{equation}\label{eq:cond1}
z\notin R.
\end{equation}
 Clearly, $R$ is $p$-periodic. Let $R'\subset R$ be a set of representatives of $R\mod p$. Then (by~\eqref{eq:cond1})
 \begin{equation}\label{eq:rlep}
 |R'|\le p-1.
 \end{equation}
Given a prime $q\le k+p$, let $y_q\in\Z$ be such that\footnote{The existence of $y_q$ follows by Dickson's condition.}
\begin{equation}\label{eq:yq}
(c+dy_q)\prod_{i=1}^k(a_i+b_iy_q)\nequiv 0\mod  q.
\end{equation}
For $r'\in R'$ choose\footnote{The set on the right-hand side of (\ref{eq:cr}) is nonempty, since the arithmetic progressions involved have pairwise coprime periods.}
\begin{equation}\label{eq:cr}
c'_{r'}\in ((r'+p\Z)\setminus (\bigcup_{q\in\PP, q\le k+p, q\neq p}(y_q+q\Z)))\cap [L,+\infty).
\end{equation}
We claim that
\begin{equation}\label{eq:Dickson}
\text{the family}\; \{f_i(x):i=1,\ldots,k\}\cup\{-c'_{r'}+x: r'\in R'\}\;\text{satisfies Dickson's condition}.
\end{equation}
Indeed, we need to prove that for every prime $q$ there exists $y\in\Z$ such that
\begin{equation}
\prod_{i=1}^kf_i(y)\prod_{r'\in R'}(y-c'_{r'})\nequiv 0\mod q.
\end{equation}

Let $q\in\PP$. If $q\le k+p$ and $q\neq p$ then $\prod_{i=1}^k(a_i+b_iy_q)\notin q\Z$ by (\ref{eq:yq}) and $-c'_{r'}+y_q\notin q\Z$ for $r'\in R'$ by (\ref{eq:cr}). We set $y=y_q$ in this case.

If $q=p$, then $z-c'_{r'}\nequiv 0\mod p\Z$ since $c_r'\in R$ (in view of \eqref{eq:cr}) while $z\notin R$ (see \eqref{eq:cond1}). By (\ref{eq:cond0}), $\prod_{i=1}^k(a_i+b_iz)\nequiv 0\mod p$ and we can set $y=z$.

Assume now that $q>k+p$. As $q$ is coprime to  $b_i$, $i=1,\ldots,k$, there exists $b'_i\in \Z$ such that $b'_ib_i\equiv 1 \mod q$ for $i=1,\ldots,k$. Since $q>k+p\ge |\Phi|+|R'|$ (see (\ref{eq:rlep})), there exists $y\in \Z$ such that $y\nequiv c'_{r'}\mod q$ for every $r'\in R'$ and $y\nequiv -b_i'a_i\mod q$ for $i=1,\ldots,k$. Then
$$
\prod_{i=1}^kf_i(y)\prod_{r'\in R'}(y-c'_{r'})\nequiv 0\mod q
$$
and the claim (\ref{eq:Dickson}) follows.

It remains to prove that the family $\{f_i(x):i=1,\ldots,k\}\cup\{-c'_{r'}+x: r'\in R'\}\cup\{c+dx\}$ does not satisfy Dickson's condition. Let $y\in\Z$. By the definition of the set $R$, if $y\notin R$ then $p|(c+dy)\prod_{i=1}^kf_i(y)$. Otherwise $y\equiv r'\mod p$ for some $r'\in R'$, hence $p|y-c'_{r'}$.
\end{proof}

\begin{Lemma}\label{lem:maxdic}
Consider finite families $\Phi=\{f_i(x)=a_i+b_ix:i=1,\ldots,k\}$ and $\Gamma=\{g_j(x)=c_j+d_jx:j=1,\ldots,k'\}$ with $a_i,b_i,c_j,d_j\in\Z$, $b_i,d_j\ge 1$, for  $i=1,\ldots,k, j=1,\ldots, k'$  and let $N\in\N$. Assume moreover that $\Phi\cap(-\Gamma\cup\Gamma)=\emptyset$.
If $\Phi$ satisfies Dickson's condition, then there exists a finite family $\Phi'$ of degree-one polynomials with integer coefficients such that $\Phi'$ satisfies Dickson's condition, $\Phi\subseteq \Phi'$ and $\Phi'\cup\{g_j(x)\}$ (for $j=1,\ldots,k'$) does not satisfy Dickson's condition. Moreover, $\Phi'\setminus\Phi$ consists of monic polynomials of the form $x-c$ with $c>N$.
\end{Lemma}

\begin{proof}
By induction on $0\le j\le k'$ we construct finite families $\Phi_{j}$ of polynomials of degree 1 with integer coefficients, satisfying Dickson's condition, such that
\begin{enumerate}
\item[(i)] $\Phi_0=\Phi$ and $\Phi_j\subseteq \Phi_{j+1}$ for $j=0,\ldots,k'-1$,
\item[(ii)] $\Phi_j$ satisfies Dickson's condition for $j=0,\ldots,k'$,
\item[(iii)] $\pm g_t(x)\notin \Phi_j$ for $j=0,\ldots,k'$ and $t=1,\ldots,k'$,
\item[(iv)] $\Phi_j\cup\{g_j(x)\}$ does not satisfy Dickson's condition for $j=0,\ldots,k'$.
\end{enumerate} Assume that $\Phi_0,\ldots,\Phi_{j}$ have been constructed and $j< k'$. We apply Lemma \ref{lem:guard} to $\Phi=\Phi_j$ and $g(x)=g_{j+1}(x)$ taking $L\ge N$ big enough to guarantee that $\pm g_t(x)\notin \Psi$ for every $t=1,\ldots,k'$. We set $\Phi_{j+1}=\Phi_j\cup\Psi$ and then $\Phi_{j+1}$ satisfies the conditions (i)-(iv). We set $\Phi'=\Phi_{k'}$. Then $\Phi'\cup\{g_{k'}\}$ does not satisfy Dickson's condition. To see that the same holds for the remaining $j$, note that already a subfamily $\Phi_j\cup\{g_j\}$ does not satisfy Dickson's condition.
\end{proof}

\begin{proposition}\label{cor:maxad}
Assume that $M<N$ and  $A\subseteq [M+1,N]$ is a $\PP$-admissible set of integers. There exists a finite set $A'\in\Z\cap[N+1,+\infty)$ such that $A\cup A'$ is $\PP$-admissible, whereas $A'\cup A\cup\{i\}$ is not, for every $i\in [M+1,N]\setminus A$.
\end{proposition}

\begin{proof}
It is enough to apply Lemma \ref{lem:maxdic} to $\Phi=\{x-a: a\in A\}$ and $\Gamma=\{x-i: i\in [M+1,N]\setminus A\}$.
\end{proof}

\begin{Lemma}\label{lem:star} Conjecture \ref{Dickson} $\Rightarrow$ Conjecture \ref{Dickson*}.
\end{Lemma}

\begin{proof} Let $\Phi=\{f_i(x)=a_i+b_ix:i=1,\ldots,k\}$ and $\Gamma=\{g_j(x)=c_j+d_jx:j=1,\ldots,k'\}$ be as in the formulation of Dickson's conjecture \ref{Dickson*}.  Let $\Phi'$ be as in Lemma \ref{lem:maxdic}.
For every $j$, the family $\Phi'\cup\{g_j(x)\}$ does not satisfy Dickson's condition, thus, there exist primes $p_1,\ldots,p_{k'}$ such that
\begin{equation}\label{eq:koniec}
p_j|g_j(y)\prod_{f\in\Phi'}f(y)\;\text{for every}\;y\in\Z
\end{equation}
for every $j=1,\ldots,k'$.
Assuming Conjecture \ref{Dickson}, the set
$$
Y_1:=\{y\in\Z: f(y)\in\PP\;\text{ for}\; f(x)\in\Phi'\}
$$
is infinite.
By (\ref{eq:koniec}), $p_j|g_j(y)$ for  every $j=1,\ldots,k'$ provided
$$
y\in Y_2:=Y_1\setminus\{y\in Y_1: \{p_1,\ldots,p_{k'}\}\cap \{f(y): f(x)\in\Phi'\}\neq\emptyset\}
$$
and $Y_2$ is cofinite in $Y_1$. As the polynomials $g_j(x)$ are not constant,  all the numbers $g_j(y)$, $j=1,\ldots,k'$, are composite for infinitely many $y\in Y_2$. At the same time, $f_1(y),\ldots,f_k(y)$ are prime for $y\in Y_2\subseteq Y_1$, because $f_1(x),\ldots,f_k(x)\in\Phi'$.
\end{proof}

Now, we turn to the first consequence of Dickson's conjecture.

\begin{Prop}\label{prop:Dickson's_and_semiprimes}
Assuming Dickson's conjecture:
for every $k,N\in \N$ and every set $A\subseteq [1,N]\cap\PP_k$, there exists $n\in\N$ such that
\begin{equation}\label{prop:formula}
[n+1,n+N]\cap\PP_{k+1}=\{n+a:a\in A\}.
\end{equation}
In particular,  $X_{\raz_{\PP_k}}\subseteq X_{\raz_{\PP_{k+1}}}$ for every $k\in\N$.

\end{Prop}

\begin{proof}
For $N\in\N$  we denote by $Q_N$ the primorial $N\#$ of $N$, that is
$$
Q_N=\prod_{p\in\PP, p\le N}p.
$$

For $k\in\N$ and $q\in\cM_{\PP_{k}}$ choose $m^{(k)}_q\in\PP_k$ such that $m^{(k)}_q|q$. If $q\in\cF_{\PP_{k}}$, we set $m^{(k)}_q=q$.
Observe that  $m^{(k)}_q=q$ for $q\in\PP_k$.
Now, fix $N,k\in\N$. Given $1\le q\le N$, we define a linear polynomial with integer coefficients
$$
h^{(k)}_{q,N}(x)=\frac{q}{m^{(k)}_q}+ \frac{Q_N^k}{m^{(k)}_q}x.
$$
Observe that
\begin{equation}\label{eq:jedynki}
h^{(k)}_{q,N}(x)=1+\frac{Q_N^k}{q}x\;\text{ if}\; q\in\cF_{\PP_{k+1}}=\cF_{\PP_{k}}\cup\PP_k.
\end{equation}

Let
$$
\Phi=\{h^{(k)}_{q,N}(x): q\in A\cup(\cF_{\PP_k}\cap[1,N])\}
$$
and
$$
\Gamma=\{h^{(k)}_{q,N}(x): q\in\cM_{\PP_k}\cap[1,N]\setminus A\}.
$$
By (\ref{eq:jedynki}), it follows that $\Phi$ satisfies Dickson's condition\footnote{Because $y\nmid 1+ay$ for any $a,y\in\Z$ such that $|y|\ge 2$.}.

Note that the polynomial $h^{(k)}_{q,N}(x)$ determines $q$: $q=\frac{Q_N^kh^{(k)}_{q,N}(0)}{h^{(k)}_{q,N}(1)-h^{(k)}_{q,N}(0)}$. It follows that
$\Phi\cap (-\Gamma\cup \Gamma)=\emptyset$.


 Assuming Conjecture \ref{Dickson}, so - in view of Lemma \ref{lem:star} -
Conjecture \ref{Dickson*}, for infinitely many $y\in\Z$,
\begin{equation}\label{eq:final}
h^{(k)}_{q,N}(y)\in\PP \Leftrightarrow h^{(k)}_{q,N}(x)\in\Phi\Leftrightarrow q\in A\cup \cF_{\PP_k}\;\text{for}\;q\in[1,N].
\end{equation}
Fix such $y\in\Z$. We claim that
\begin{equation}\label{eq:final2}
m^{(k)}_qh^{(k)}_{q,N}(y)=q+Q_N^ky\in\PP_{k+1} \Leftrightarrow q\in A\;\text{for}\; q\in[1,N].
\end{equation}
Indeed, if $q\in A\subseteq \PP_k$, then  $m^{(k)}_q=q\in\PP_k$ and, as $h^{(k)}_{q,N}(y)\in\PP$ by (\ref{eq:final}), $m^{(k)}_qh^{(k)}_{q,N}(y)\in\PP_{k+1}$.
If $q\notin A$, then one of the following conditions hold:
\begin{enumerate}
\item[(i)] $q\in\PP_k\setminus A$,
\item[(ii)] $q\in \cF_{\PP_k}$,
\item[(iii)] $q\in \cM_{\PP_{k+1}}$.
\end{enumerate}
In the case (i), $h^{(k)}_{q,N}(x)\in\Gamma$, thus $h^{(k)}_{q,N}(y)$ is composite and $m^{(k)}_qh^{(k)}_{q,N}(y)=qh^{(k)}_{q,N}(y)\in\cM_{\PP_{k+2}}$. In the case (ii), $m^{(k)}_q=q$ and
$h^{(k)}_{q,N}(x)\in\Phi$, hence $h^{(k)}_{q,N}(y)\in\PP$ and $qh^{(k)}_{q,N}(y)\in \cf_{\PP_{k+1}}$. If (iii) holds, then $h^{(k)}_{q,N}(x)\in\Gamma$, $h^{(k)}_{q,N}(y)$ is composite, hence $m^{(k)}_qh^{(k)}_{q,N}(y)\in\cM_{\PP_{k+2}}$. In each of the cases (i), (ii), (iii), $m^{(k)}_qh^{(k)}_{q,N}(y)\notin\PP_{k+1}$.

It follows by (\ref{eq:final2}) that $n=Q_N^ky$ satisfies the condition (\ref{prop:formula}).

\end{proof}

The following proposition will be applied to obtain a lower bound for the complexity of $\raz_{\PP_2}$.

\begin{Prop}\label{prop:herk}
Assuming Dickson's conjecture, for every  $M,N\in\N$ such that $\sqrt{N}\le M<N
$, $N-M<\sqrt{N}$ and  $[\sqrt{N}]\notin \P$ and for  every subset
$$
A\subseteq \PP_2\cap[M+1,N]
$$
there exists $n\in \N$ such that
\begin{equation}\label{prop:teza3}
\PP_2\cap [n+1,n+N-M]=\{n+a:a\in A\}.
\end{equation}
\end{Prop}

\begin{proof}
First note that
\begin{equation}\label{duzydzielnik}
q\in\PP_2\cap[M+1,N]\Rightarrow q\;\text{has a prime divisor}\;p>\sqrt{N}.
\end{equation}
Indeed, it assume for a contradiction that $M<p_1 p_2\le N$ for some primes $p_1,p_2\le \sqrt{N}$. Then $\sqrt{N}\ge p_1>\frac{M}{\sqrt{N}}\ge  \sqrt{N}-1$ because $M>N-\sqrt{N}$. This means that $[\sqrt{N}]=p_1\in\P$, contrary to our assumption. Let
$$
Q=\sqrt{N}\#:=\prod_{p\in\PP, p\le\sqrt{N}}p.
$$
Given a  number $q\in \cM_{\PP_2}$, let $m_q$ be the minimal prime divisor of $q$. Then
$m_q\le\sqrt{N}$ for $q\in \cM_{\PP_2}\cap [M+1,N]$.
Let
$$
h_q(x)=\frac{q}{m_q}+\frac{Q^2}{m_q}x
$$
for $q\in\cM_{\PP_2}\cap  [M+1,N]$.

For $q\in\PP\cap  [M+1,N]$, we set
$$
h_q(x)=q+Q^2x.
$$

Let
$$
\Phi=\{h_q(x):q\in A\cup(\PP\cap[M+1,N])\}
$$
and
$$
\Gamma=\{h_q(x):q\in [M+1,N]\setminus (A\cup\PP)\}.
$$
We claim that $\Phi$ satisfies the assumptions (a) and (b) of Lemma \ref{lem:condition}, hence it satisfies Dickson's condition.

(a) If $q\in A$, then $\frac{q}{m_q}$ is a prime and  $\frac{q}{m_q}>\sqrt{N}$ thanks to (\ref{duzydzielnik}). Hence $\frac{q}{m_q}$ is coprime to $\frac{Q^2}{m_q}$. Clearly, $q$ is  coprime to $Q^2$ for $q\in\PP\cap  [M+1,N]$, because $q>M\ge \sqrt{N}$. Thus (a) of Lemma \ref{lem:condition} is satisfied.

(b) If $p\le \sqrt{N}$, then $p|Q^2$ and  for every $q\in A$, as $m_q\in\PP$, $p|\frac{Q^2}{m_q}$.
If $p>\sqrt{N}$ then $p>|\Phi|$, because $\sqrt{N}> N-M\ge |\Phi|$ by assumption. Thus (b) is satisfied.
The claim follows.

Observe that the polynomial $h_q(x)$ determines $q$: $q=\frac{Q^2h_q(0)}{h_q(1)-h_q(0)}$. Therefore
$\Phi\cap(-\Gamma\cup\Gamma)=\emptyset$.

Assuming Conjecture \ref{Dickson}, so - in view of Lemma \ref{lem:star} -
Conjecture \ref{Dickson*}, for infinitely many $y\in\Z$, for $q\in [M+1,N]$, we have
\begin{equation}\label{eq:final1}
h_{q}(y)\in\PP \Leftrightarrow h_{q}(x)\in\Phi\Leftrightarrow q\in A\cup (\PP\cap [M+1,N]).
\end{equation}
Therefore, since $m_q\in\PP$ if $q\in\cM_{\PP_2}$,
\begin{equation}
q+Q^2y=m_qh_{q}(y)\in\PP_2 \;\text{for}\; q\in A.
\end{equation}
Moreover,
\begin{equation}
q+Q^2y=h_q(y)\in\PP\;\text{for}\; q\in\PP\cap[M+1,N].
\end{equation}
Finally, since $h_q(y)$ is composite for $q\in [M+1,N]\setminus (A\cup\PP)$,
\begin{equation}
q+Q^2y=m_qh_{q}(y)\in\cM_{\PP_{3}}\;\text{for}\; q\in [M+1,N]\setminus (A\cup\PP).
\end{equation}
We have shown that for $q\in[M+1,N]$: $q+Q^2y\in\PP_2$ if and only if $q\in A$.
It follows that $n=Q^2y$ satisfies (\ref{prop:teza3}).

\end{proof}

\begin{Cor}\label{cor:comppk} Assuming Dickson's conjecture \ref{Dickson},
\begin{equation}\label{cor:cpxpk}
{\rm cpx}_{\raz_{\PP_2}}(n)\gg (2+o(1))^{\frac{n\log\log n}{2\log n}}.
\end{equation}
\end{Cor}

\begin{proof} By (\ref{NT:pnt61}),
\begin{equation}\label{eq:cpxpk1}
|\PP_2\cap [n^2-n+2,n^2]|=\left(1+o(1)\right)\left(\frac{n^2\log\log n^2}{\log n^2}-\frac{(n^2-n+1)\log\log (n^2-n+1)}{\log (n^2-n+1)}\right)=\left(1+o(1)\right)\left(\frac{n\log\log n}{2\log n}\right).
\end{equation}

By Proposition \ref{prop:herk}, if $n\notin \PP$,  every subset of $\PP_2\cap [n^2-n+2,n^2]$ is the support of a block of length $n-1$ appearing on $\raz_{\PP_2}$ and the assertion follows.
\end{proof}

\subsection{Complexity for the subshift of semi-primes (an upper bound in Theorem~\ref{t:SP})}\label{s:DC3}

We will consider a proof that is different from the prime case in \S\ref{s:TTao}. Let
$$
B_{n,N}:=(\raz_{\PP_{2}}(n+1),\ldots,\raz_{\PP_{2}}(n+N)), \;C_N:=\{B_{n,N}:\:n\geq0\}\text{ and}$$
$$O_{n,N}:=\{1\leq i\leq N:\: \raz_{\PP_{2}}(n+i)=1\}.$$

Observe that
\begin{equation}\label{ones1}
|O_{n,N}|=\sum_{\substack{n<pq\leq n+N\\p,q\in\mathbb{P}}}1=\sum_{\substack{n<m\leq n+N\\\omega(m)=2}}1+\sum_{\substack{\sqrt{n}<p\leq \sqrt{n+N}\\p\in\mathbb{P}}}1
\end{equation}
where $\omega:n\mapsto\sum_{p|n}1$. If we suppose that $n>N$, the first summation in \eqref{ones1} can be bounded by means of estimation \eqref{Tud}. Furthermore, as
\begin{equation*}
\sum_{\substack{\sqrt{n}<p\leq \sqrt{n+N}\\p\in\mathbb{P}}}1\leq\sum_{\substack{\sqrt{n}<p\leq \sqrt{n}+\sqrt{N}\\p\in\mathbb{P}}}1,
\end{equation*}
we can use the Brun--Titchmarsh theorem \eqref{Brun-Tit}, to derive
\begin{equation}\label{ones2}
|O_{n,N}|\leq\frac{23N\log(\log(N))}{\log(N)}(1+o(1))+\frac{4\sqrt{N}}{\log(N)}=\frac{23N\log(\log(N))}{\log(N)}(1+o(1)).
\end{equation}
Thus, \eqref{ones2} exhibits a bound for $|O_{n,N}|$ that is independent of $n$, as long as $n>N$.

On the other hand, let $L=L_N=\left[\frac{23N\log(\log(N))}{\log(N)}(1+o(1))\right]$. We have that
\begin{equation}\label{binom}
|\{B_{n,N}:\:n>N\}|\leq\sum_{k=0}^L\binom{N}{k},
\end{equation}
where the right-hand side corresponds to the number of blocks of length $N$ with at most $L$ coordinates equal to $1$ (the remaining coordinates being $0$'s). Hence, by \eqref{binom_partial_sum}, and recalling that the function $t\leq\frac{N}{e}\mapsto\left(\frac{N}{t}\right)^t$ is increasing, we conclude that for $N$ big enough
\begin{align}\label{est_Bn}
|\{B_{n,N}:\:n>N\}|\leq\left(\frac{eN}{L}\right)^L&\leq\left(\frac{e\log(N)(1+o(1))}{23\log(\log(N))}\right)^{\frac{23N\log(\log(N))}{\log(N)}(1+o(1))}\nonumber\\
&\leq\left(\frac{e\log(N)(1+o(1))}{23\log(\log(N))}\right)^{\frac{23N\log(\log(N))}{\log(N)}(1+o(1))}\nonumber\\
&=4^{\frac{23N\log(\log(N))}{\log(4)\log(N)}(1+o(1))\log\left(\frac{e\log(N)(1+o(1))}{23\log(\log(N))}\right)}\nonumber\\
&\leq 4^{\frac{23N\log^2(\log(N))}{\log(4)\log(N)}(1+o(1))}.
\end{align}
Finally, we derive from \eqref{est_Bn} that
\beq\label{semipr2}
|C_N|\leq |\{B_{n,N}:\:n\leq N\}|+|\{B_{n,N}:\:n>N\}|\leq N+4^{\frac{23N\log^2(\log(N))}{\log(4)\log(N)}(1+o(1))}
\leq(4+o(1))^{\frac{23N\log^2(\log(N))}{\log(4)\log(N)}},\eeq
which gives the claimed upper bound.

\begin{remark} \label{primebound2} A similar approach could have been used for \S\ref{s:TTao}. Indeed, by the Brun-Titchmarsh theorem \eqref{Brun-Tit}, the maximum number of coordinates $1$ that we can distribute inside a prime indicator vector of length $N$ is $\frac{2N}{\log(N)}$. Therefore, we obtain the following upper bound for the number of possible patterns of lenght $N$ of the indicator of the primes  (cf.~\eqref{eq:an})
\beq\label{pr2}
|A_N|\leq \left(\frac{e\log(N)}{2}\right)^{\frac{2N}{\log(N)}}=4^{\frac{2N\log(\log(N))}{\log(4)\log(N)}(1+o(1))}\leq(4+o(1))^{\frac{2N\log(\log(N))}{\log(4)\log(N)}},\eeq
which is worse than the one obtained in \eqref{tao1}.
\end{remark}
\section{Appendix}\label{sec:NT}

In the appendix we collect some classical facts from number theory, which we repeatedly apply in the paper. Here $p$ is always assumed to be prime.

\textbf{Primes}.

The Prime Number Theorem \cite[\S 6.2]{Mont-Vaug2} asserts that
\begin{equation}\label{NT:pnt1}
\pi(n):=\sum_{p<n}1=(1+o(1))\frac{n}{\log n}.
\end{equation}
or equivalently, that
\begin{equation}\label{NT:pnt4}
\prod\limits_{p<n}p=\exp((1+o(1))n).
\end{equation}

As a consequence, if $p_n$ denotes the $n$-th prime then \cite[\S 6.2.1, Ex. 5]{Mont-Vaug2}
\begin{equation}\label{NT:pnt2}
p_n=(1+o(1))n\log n.
\end{equation}

Moreover, the prime number estimates of Chebyshev and Mertens \cite[\S 2.2]{Mont-Vaug2}, weaker than the Prime Number Theorem, suffice to state that
\begin{equation}\label{NT:pnt3}
\sum\limits_{p<n}\frac1p=(1+o(1))\log\log n,
\end{equation}
as well as Mertens' Theorem
\begin{equation}\label{NT:mer}
\prod_{p\in\PP:p< x}\left(1-\frac1p\right)=\frac{e^{-\gamma}}{\log x}(1+o(1)).
\end{equation}

\textbf{Sieves}.

One of the main tools in the theory of sieves is the Large Sieve \cite{Bom}.
\vspace{2ex}

The Large Sieve - Analytic form: {\em consider the trigonometric polynomial
$$S(x)=\sum_{n=M}^{M+N} a_ne^{2\pi inx}$$
of length~$N$, where $a_n\in\C$. Let $\{x_1,\ldots,x_R\}$ be a set of $R$ reals that are $\delta$-well spaced (i.e., $\|x_i-x_j\|:=\min_{\ell\in\Z}|(x_i-x_j)-\ell|\geq\delta>0$ for $1\leq i\neq j\leq R$). Then, we have the following inequality
\begin{equation}\label{LS2}
\sum_{j=1}^R |S(x_j)|^2\leq (N+1/\delta)\sum_{n=M+1}^{M+N}|a_n|^2.
\end{equation}}

A direct application of the above result is the following.
\vspace{2ex}

 The Large Sieve Inequality: {\em Let $C$ be a set of integers contained in $[M+1,M+N]$ which avoids $\omega(p)$ residue classes modulo $p$ for each prime and let $R>0$.
Then},
\begin{equation}\label{LS1}
|C|\leq \frac{N+R^2}{\sigma(R)},
\end{equation}
where $\sigma(R)=\sum_{q\leq R}\mob^2(q)\prod_{p|q}\frac{\omega(p)}{p-\omega(p)}$.

Closely related to Sieve theory, is the following result (refer to the Selberg sieve \cite{CojM}[\S 4] ).
\begin{equation}\label{NT:pnt5}
\sum\limits_{q\leq n}\frac{\mu^2(q)}{\varphi(q)}=(1+o(1))\log n.
\end{equation}
This result relies upon a convolution identity. See \cite{RA13} and \cite{1ZA20} for further details.

A weaker result, but sufficient for the estimation in~\eqref{osz}, is that
$\sum\limits_{q\leq n}\frac{\mu^2(q)}{\varphi(q)}>\log n$;
see \cite[Eq. (3.18)]{Mont-Vaug2}.

\textbf{Numbers that are products of a fixed number of primes}.

By recalling  \cite[\S 7.4]{Mont-Vaug2}, we have that, for fixed $k$,

\begin{equation}\label{NT:pnt61}
\sum\limits_{q<n}\raz_{\PP_k}(q)=(1+o_k(1))\frac{n\log^{k-1}\log n}{(k-1)!\log^{k-1}n},
\end{equation}
as a consequence, we have by summation by parts that
\begin{equation}\label{NT:pnt6}
\sum\limits_{q<n}\frac{\raz_{\PP_2}(q)}{q}=\frac{1}{2}\log^2\log(n))(1+o(1))
\end{equation}

\textbf{Uniform bounds on primes and semi-primes in intervals}.

By means of the Large Sieve, one can find uniform bounds for almost primes.
The Brun-Titchmarsh theorem \cite[Thm. 2]{Mont-Vaug} asserts in particular that for all $X>0$, $Y>1$,
\begin{equation}\label{Brun-Tit}
\pi(X+Y)-\pi(X)<\frac{2Y}{\log(Y)}.
\end{equation}

Related to semi-primes, Tudesq \cite{Tudesq} gives the following uniform bound.
\begin{equation}\label{Tud}
\sum_{\substack{X<m\leq X+Y\\ \omega(m)=2}}1\leq\frac{23Y\log(\log(Y))}{\log(Y)}(1+o(1)),
\end{equation}
provided that $2\leq Y\leq X$.

Refer to \cite{CKKT} and \cite{Tudesq} for further discussion about uniform bounds for almost primes in intervals in arithmetic progressions.

\textbf{Partial sums of binomial coefficients}.

Let $N$ be a natural number. Then for any integer $0\leq L\leq N$, we have
\begin{equation}\label{binom_partial_sum}
\sum_{k=0}^L\binom{N}{k}\leq\sum_{k=0}^L\frac{N^k}{k!}=\sum_{k=0}^L\frac{L^k}{k!}\frac{N^k}{L^k}\leq\frac{N^L}{L^L}\sum_{k=0}^\infty\frac{L^k}{k!}=\left(\frac{eN}{L}\right)^L.
\end{equation}

Faculty of Mathematics and Computer Science, Nicolaus Copernicus University, Chopin street 12/18, 87-100 Toru\'n, Poland

\begin{thebibliography}{99}
\bibitem{Ab-Le-Ru}H.\ El Abdalaoui, M. Lema\'nczyk, T. de la Rue, {\em A dynamical point of view on the set of $\mathcal{B}$-free integers},  International Mathematics Research Notices {\bf 16} (2015), 7258-7286.
\bibitem{Ben} M. A. Bennet, {\em On some exponential equations of S. S. Pillai.} Canad. \ J. \ Math. {\bf 53} (2001), no. 5, 897--922.
\bibitem{Bom} E. Bombieri, {\em Le grand crible dans la th\'eorie analytique des nombres} (in French), Soci\'et\'e math\'ematique de France, 1974.
\bibitem{CKKT} Chan T., Kwok-Kwong C., Tsang K.M., {\em An Extension to the Brun--Titchmarsh Theorem} The Quarterly Journal of Mathematics
Quart. J. Math. 00 (2010), 1--16.
\bibitem{CojM} A.M.\ Cojocaru, M.\ Ram Murty, {\em An Introduction to Sieve Methods and Their Applications}, London Mathematical Society, Student Texts 66.
\bibitem{Cy-Kr} V.\ Cyr, B. Kra, private communication.
\bibitem{Dickson} L. E. Dickson, {\em A new extension of Dirichlet's theorem on prime
numbers}, Messenger of Mathematics, 33 (1904), 155--161.

\bibitem{Do}T.\ Downarowicz, {\em Survey of odometers and Toeplitz flows}, Algebraic and topological dynamics,
7-37, Contemp. Math., 385, Amer. Math. Soc.., Providence, RI, 2005.
\bibitem{DoMON}T.\ Downarowicz, {\em Entropy in Dynamical Systems}, New mathematical monographs:18, Cambridge University Press 2011.
\bibitem{Dy-Ka-Ku-Le}A. Dymek, S.\ Kasjan, J.\ Ku\l aga-Przymus, M. Lema\'nczyk, {\em $\mathscr{B}$-free sets and dynamics}, Trans.\ Amer.\ Math.\ Soc.\ {\bf 370} (2018), 5425--5489.
\bibitem{Fr-Ka-Le}K. Fr\c{a}czek, A. Kanigowski, M. Lema\'nczyk, {\em Prime number theorem for regular Toeplitz systems}, Ergodic Theory Dynam. Systems {\bf 42} (2022), 1446-1473.
\bibitem{Hal}H.\ Halberstam, K.F.\ Roth, {\em Sequences}, Springer-Verlag, New York Inc, 1983.
\bibitem{Ha}R.R. Hall, {\em Sets of multiplies}, vol.\ 118 of Cambridge Tracts in Mathematics, Cambridge University Press, Cambridge, 1996.
\bibitem{Ha-Li} G. H. Hardy, J. E.  Littlewood, {\em Some Problems of `Partitio Numerorum.' III. On the Expression of a Number as a Sum of Primes}, Acta Math. {\bf 44} (1922),  1--70.
\bibitem{Ka-Ke-Le}S.\ Kasjan, G.\ Keller, M.\ Lema\'nczyk, {\em Dynamics of $\mathscr{B}$-free sets: a view through the window},  Int. Math. Res. Not. IMRN 2019, no. 9, 2690-2734.
\bibitem{Ke}G. Keller, {\em Tautness of sets of multiples and applications to $\mathcal{B}$-free systems}, Studia
Math. 247 (2019), no. 2, 205-216. Corrigendum Studia Math. 258 (2021), no. 2,
235-237.
\bibitem{Ke1}G. Keller, {\em Generalized heredity in $\mathcal{B}$-free systems},  Stoch. Dyn. 21 (2021), no. 3, Paper No. 2140008, 19 pp.
\bibitem{Ke2}G. Keller, {\em Irregular $\mathcal{B}$-free Toeplitz sequences via Besicovitch's construction of sets of multiples without density}, arXiv 2101.00655.
\bibitem{Ku-Le(jr.)} J. Ku\l aga-Przymus, M.D.\ Lema\'nczyk (jr.), {\em Hereditary subshifts whose measure of
maximal entropy does not have the Gibbs property}, Coll. Math.  166 (2021), no. 1, 107-127.
\bibitem{Ku-Le-We} J.\ Ku\l aga-Przymus, M.\ Lema\'nczyk, B. Weiss, {\em Invariant measures for ${\cal B}$-free systems},  Proc. Lond. Math. Soc. (3) {\bf 110} (2015),  1435--1474.
\bibitem{Le-Ri-Se} M.\ Lema\'nczyk, C. Richard, D. Sell, {\em On the Garden of Eden theorem for $\mathscr{B}$-free subshifts}, to appear in Israel J. Math., arXiv:2106.14673
\bibitem{Ma-Ri}C. Mauduit, J. Rivat, {\em Prime numbers along Rudin-Shapiro sequences}, J. Eur. Math. Soc. 17
(2015), 2595-2642.
\bibitem{Meyer} Y. Meyer, {\em Ad\`eles et s\'eries trigonom\'etriques sp\'eciales}, (French), Ann. of Math. (2) 97 (1973),
171-186.
\bibitem{Mont-Vaug} H. L. Montgomery, R. C. Vaughan, {\em The Large Sieve},
    Mathematika, 20(2) (1973), 119 - 134.

\bibitem{Mont-Vaug2} H. L. Montgomery, R. C. Vaughan, {\em Multiplicative number theory: I. Classical theory.} Cambridge University Press (2007).

\bibitem{Mu} C. M\"ullner, {\em Automatic sequences fulfill the Sarnak conjecture}, Duke Math. J. 166 (2017),
3219-3290.

\bibitem{Na}W.\ Narkiewicz, {\em Teoria liczb}, PWN.


\bibitem{Pa} R. Pavlov, {\em Some counterexamples in topological dynamics}, Ergodic Theory Dynam. Systems
28 (2008), 1291-1322.

\bibitem{RA13}
O. Ramar\'e, Akhilesh P., {\em Explicit averages of non-negative multiplicative functions: going beyond the main term.} Colloquium Mathematicum, 2017, 147 (2), pp. 275 - 313.

\bibitem{Sa} P. Sarnak, {\em Three lectures on the M\"obius function, randomness and dynamics},
http://publications.ias.edu/sarnak/
\bibitem{Sch61} A. Schinzel, {\em Remarks on the paper ``Sur certaines hypotheses concernant les nombres premiers''}, Acta Arithmetica {\bf 7}, No. 1, (1961), 1--8.
\bibitem{Tao} T. Tao, private communication.
\bibitem{Tenenbaum} G. Tenenbaum, {\em Introduction to Analytic and Probabilistic Number Theory}, Graduate Studies in Mathematics
Volume: 163; AMS (2015).
\bibitem{Toth} L. T\'oth, {\em On the Asymptotic Density of Prime k-tuples and a Conjecture of Hardy and Littlewood}, CMST (Comp. Meth. in Sci. and Tech. 25(3) 143-148 (2019).

\bibitem{Tudesq} C. Tudesq, {\em \'Etude de la Loi Locale de $\omega(n)$ Dans de Petits Intervalles}, (French), The Ramanujan Journal, 4, 277-290, 2000.
\bibitem{Zhang} S. Zhang, {\em Notes on Dickson's Conjecture}, (2009)  arXiv:0906.3850v3.

\bibitem{1ZA20}
S. Zuniga Alterman, {\em Explicit averages of square-free supported functions: beyond the convolution method.} Colloquium Mathematicum {\bf 168} (2022), 1-23.

 \end{thebibliography}
\end{document}